\documentclass{amsart}

\usepackage{amssymb, amsmath}
\usepackage{graphicx}

\newtheorem{theorem}{Theorem}[section]
\newtheorem{lemma}[theorem]{Lemma}
\newtheorem{claim}[theorem]{Claim}
\newtheorem{corollary}[theorem]{Corollary}
\newtheorem{proposition}[theorem]{Proposition}

\theoremstyle{definition}
\newtheorem{definition}[theorem]{Definition}

\newtheorem{question}[theorem]{Question}

\theoremstyle{remark}
\newtheorem{remark}[theorem]{Remark}

\numberwithin{equation}{section}

\begin{document}

\title{Right-angled Artin groups and full subgraphs of graphs}


\author{Takuya Katayama}
\address{Department of Mathematics, Hiroshima University, 1-3-1 Kagamiyama, Higashi-Hiroshima, 739-8526, Japan}
\curraddr{}
\email{tkatayama@hiroshima-u.ac.jp}
\thanks{}



\date{}

\dedicatory{}

\begin{abstract}
For a finite graph $\Gamma$, let $G(\Gamma)$ be the right-angled Artin group defined by the complement graph of $\Gamma$. 
We show that, for any linear forest $\Lambda$ and any finite graph $\Gamma$, $G(\Lambda)$ can be embedded into $G(\Gamma)$ if and only if $\Lambda$ can be realised as a full subgraph of $\Gamma$. 
We also prove that if we drop the assumption that $\Lambda$ is a linear forest, then the above assertion does not hold, namely, for any finite graph $\Lambda$, which is not a linear forest, there exists a finite  graph $\Gamma$ such that $G(\Lambda)$ can be embedded into $G(\Gamma)$, though $\Lambda$ cannot be embedded into $\Gamma$ as a full subgraph. 
\end{abstract}

\maketitle

\section{Introduction and statement of results} 
\label{intro}

Let $\Gamma$ be a finite simplicial graph (abbreviated a finite graph), with the vertex set 
$V(\Gamma) = \{ v_1 , v_2 , \ldots , v_n \} $ and the edge set $E(\Gamma)$. 
In this paper, we denote an element of $E(\Gamma)$ by $[v_i, v_j]$. 
Then the {\it right-angled Artin group} (abbreviated {\it RAAG}) on $\Gamma$ is the group given by the  following presentation:
$$
A(\Gamma) = \langle \ v_1 , v_2 , \ldots , v_n  \ \mid \  v_i v_j v_i^{-1} v_j^{-1} = 1  \ \mbox{if} \ [ v_i,v_j ] \in E(\Gamma) \ \rangle
.$$ 
In this paper, we denote $G(\Gamma)$ to be $A(\Gamma^c)$, where $\Gamma^c$ is the {\it complement} or the {\it opposite graph} of $\Gamma$, namely, $\Gamma^c$ is the graph consisting of the vertex set $V(\Gamma^c) = V(\Gamma)$ and the edge set $E(\Gamma^c)= \{ [u, v ] \mid u,v \in V(\Gamma), \ [u , v ] \notin E(\Gamma) \}$.
Namely, 
$$
G(\Gamma) = \langle \ v_1 , v_2 , \ldots , v_n  \ \mid \ v_i v_j v_i^{-1} v_j^{-1} = 1  \ \mbox{if} \ [ v_i,v_j ] \notin E(\Gamma) \ \rangle
.$$

The following question was raised by S. Kim and T. Koberda \cite[Question 1.1]{Kim-Koberda-1} (see also \cite{Crisp-Sageev-Sapir}). 

\begin{question}
Is there an algorithm to decide whether there exists an embedding between two given RAAGs?
\label{CSS}
\end{question}

Several studies have demonstrated that the embeddability of RAAGs can be understood via certain graph theoretical concepts (e.g. \cite{Casals-Ruiz}, \cite{Casals-Ruiz-Duncan-Kazachkov}, \cite{Kim}, \cite{Kim-Koberda-1}, \cite{Kim-Koberda-2} and \cite{Lee-Lee}). 
In fact, theorems due to Kim and Koberda \cite{Kim-Koberda-1} state that the following for any finite graphs $\Lambda$ and $\Gamma$. 
\begin{enumerate}
 \item[$\bullet$] Any embedding of $\Lambda$ into the ``extension graph" of $\Gamma$ gives rise to an embedding of $A(\Lambda)$ into $A(\Gamma)$. 
 \item[$\bullet$] Any embedding of $A(\Lambda)$ into $A(\Gamma)$ gives rise to an embedding of $\Lambda$ into the ``clique graph" of the ``extension graph" of  $\Gamma$. 
 \end{enumerate}
These studies suggest us that certain graph theoretical tools can be useful to study Question \ref{CSS} to which we do not know the answer. 
This paper mainly concentrates on giving a complete answer to Question \ref{CSS-induced-condition} (see below), which concerns possibly the simplest graph theoretical obstruction to the existence of embeddings between RAAGs. 

In order to state Question \ref{CSS-induced-condition}, we prepare some terminology. 
A subgraph $\Lambda$ of a graph $\Gamma$ is said to be  {\it full} or {\it induced} if $E(\Lambda)$ contains every $e \in E(\Gamma)$ whose end points both lie in $V(\Lambda)$. 
Note that the full subgraph $\Lambda$ of $\Gamma$ is uniquely determined by its vertex set $V':= V(\Lambda) \subset V(\Gamma)$. 
So we denote $\Lambda$ by $\Gamma[V']$ and say that $V'$ induces $\Lambda = \Gamma[V']$. 
Besides we denote by $\Lambda \leq \Gamma$ if there exists an full subgraph of $\Gamma$, which is isomorphic to $\Lambda$. 
We denote by $G \hookrightarrow H$ if there exists an injective homomorphism (abbreviated an {\it embedding}) from a group $G$ into a group $H$. 
It is well-known that the implication $\Lambda \leq \Gamma \Rightarrow G(\Lambda) \hookrightarrow G(\Gamma)$ is always true. 
However, in general, the converse implication $G(\Lambda) \hookrightarrow G(\Gamma) \Rightarrow \Lambda \leq \Gamma $ is false. 
In fact $G(K_3) = F_3 \hookrightarrow F_2 = G(K_2)$, though $K_3 \not\leq K_2$, where $K_2$ (resp. $K_3$) denotes the complete graph on $2$ (resp. $3$) vertices. 
So, we can ask the following natural question. 

\begin{question}
Which finite graph $\Lambda$ satisfies the following property $(*)$?

$(*)$ For any finite graph $\Gamma$, $G(\Lambda) \hookrightarrow G(\Gamma)$ implies $\Lambda \leq \Gamma$. 
\label{CSS-induced-condition}
\end{question}

Before stating our results, we define some symbols and terminology of graphs. 

\begin{enumerate}
 \item[$\bullet$] $K_n$: the {\it complete graph} on $n$ vertices, i.e., $V(K_n)$ has $n$ elements and each pair of vertices in $V(K_n)$ spans an edge.   
 \item[$\bullet$] $P_n$: the {\it path graph}  on $n$ vertices, i.e., $P_n$ is the connected graph consisting of $(n-2)$ vertices of degree $2$ and two vertices of degree $1$. 
  A {\it linear forest} is the disjoint union of path graphs.  
 \item[$\bullet$] $C_n$: the {\it cyclic graph} on $n \ (\geq 3)$ vertices, i.e., $C_n$ is the connected graph consisting of $n$ vertices of degree $2$. 
\end{enumerate}

The main theorem of this paper is the following. 
 
\begin{theorem}
Let $\Lambda$ be a finite graph. 
\begin{enumerate}
 \item[(1)] If $\Lambda$ is a linear forest, then $\Lambda$ has property $(*)$, namely, for any finite graph $\Gamma$, $G(\Lambda) \hookrightarrow G(\Gamma)$ implies $\Lambda \leq \Gamma$. 
 \item[(2)] If $\Lambda$ is not a linear forest, then $\Lambda$ does not have property $(*)$, namely, there exists a finite graph $\Gamma$ such that $G(\Lambda) \hookrightarrow G(\Gamma)$, though $\Lambda \not\leq \Gamma$.  
\end{enumerate}
\label{Main-theorem}
\end{theorem}

Theorem \ref{Main-theorem}(1) generalises the following well-known fact: for any finite graph $\Gamma$, $ \mathbb{Z}^n = A(K_n) \hookrightarrow A(\Gamma)$ implies $K_n \leq \Gamma$ (see e.g.~\cite{Charney-Vogtmann}). 
In terms of the opposite convention it says that, for any finite graph $\Gamma$, $G(K_n^c) \hookrightarrow G(\Gamma)$ implies $K_n^c \leq \Gamma$. 
Hence, $K_n^c$ has our property $(*)$ and the graph $K_n^c$ is in fact a linear forest ($n$ isolated vertices). 
Theorem \ref{Main-theorem}(1) also generalises the  result of Kim-Koberda \cite{Kim-Koberda-1}, which states that the linear forests  $P_3^c = P_1 \sqcup P_2$ (the symbol $\sqcup$ means the disjoint union), $P_4^c = P_4$ and $C_4^c = P_2 \sqcup P_2$ have property $(*)$. 

As a consequence of Theorem \ref{Main-theorem}(1) and a result of Kim \cite{Kim}, we obtain the following result concerning embeddability between RAAGs on finite graphs whose underlying spaces are connected $1$-manifolds. 

\begin{theorem}
Let $m$ and $n$ be positive integers. 
Then the following hold. 
\begin{enumerate}
 \item[(1)] $G(P_m) \hookrightarrow G(P_n)$ if and only if $m \leq n$. 
 \item[(2)] $G(C_m) \hookrightarrow G(C_n)$ if and only if $m \leq n$. 
 \item[(3)] $G(P_m) \hookrightarrow G(C_n)$ if and only if $m+1 \leq n$. 
\end{enumerate}
For embedding $G(C_m)$ into $G(P_n)$, we have the following. 
\begin{enumerate}
 \item[(4-1)] $G(C_3) \hookrightarrow G(P_n)$ if and only if $2 \leq n$. 
 \item[(4-2)] $G(C_4) \hookrightarrow G(P_n)$ if and only if $3 \leq n$. 
 \item[(4-3)] Suppose that $5 \leq m$. 
 If $G(C_m) \hookrightarrow G(P_n)$, then $m - 1 \leq n$. 
\end{enumerate}
\label{Anti-lines-anti-cycles}
\end{theorem}

The ``only if" parts of Theorem \ref{Anti-lines-anti-cycles}(1), (2), (3), (4-1), (4-2) and (4-3) are direct consequences of Theorem \ref{Main-theorem}(1). 
The ``if" part of Theorem \ref{Anti-lines-anti-cycles}(2) is nothing other than the result of Kim \cite[Corollary 4.3]{Kim}. 
We will give a proof to the result by using subdivision technique (Lemma \ref{Complementary-subdivision}). 
\begin{remark}
Theorem \ref{Anti-lines-anti-cycles}(4-3) is not best possible. 
In fact, the result of C. Droms \cite[Theorem 1]{Droms} implies that $G(C_5)$ cannot be embedded into $G(P_4)$ though $(m,n)=(5,4)$ satisfies the inequality. 
Moreover, E. Lee and S. Lee \cite{Lee-Lee} proved that $G(C_m) \hookrightarrow G(P_n)$ if $2m-2 \leq n$. 
\end{remark}

We note that Theorem \ref{Main-theorem}(1) has an application to the existence of embeddings of RAAGs into mapping class groups. 
Let $\Sigma_{g,n}$ be the orientable surface of genus $g$ with $n$ punctures and $\mathcal{C}(\Sigma_{g,n})$ the curve graph on $\Sigma_{g,n}$. 
The mapping class group of $\Sigma_{g,n}$ is defined by 
$$\mathcal{M}(\Sigma_{g,n}) = \pi_0(\mathrm{Homeo}^{+} (\Sigma_{g,n})),$$ 
namely, $\mathcal{M}(\Sigma_{g,n})$ is the group of orientation-preserving self-homeomorphisms of $\Sigma_{g,n}$ which preserve the set of punctures, up to isotopy. 
Assume that $\chi(\Sigma_{g,n}) = 2 -2g - n  < 0$. 
Under this setting, Koberda \cite{Koberda} proved that if $\Lambda$ is a finite full subgraph of $\mathcal{C}(\Sigma_{g,n})$, then $A(\Lambda) \hookrightarrow \mathcal{M}(\Sigma_{g,n})$. 
Regarding this result, Kim-Koberda proposed the following. 

\begin{question}\cite[Question 1.1]{Kim-Koberda-3}
Is there an algorithm to decide whether there exists an embedding of a given RAAG into the mapping class group of a given compact orientable surface?
\label{RAAG_Map}
\end{question}

Motivated by this question, we deduce a partial converse of Koberda's embedding theorem above, as an immediate consequence of Theorem \ref{Main-theorem}(1)  in this paper. 
\begin{corollary}
Suppose that $\Lambda$ is the complement of a linear forest. 
Then $A(\Lambda) \hookrightarrow \mathcal{M}(\Sigma_{g,n})$ implies $\Lambda \leq \mathcal{C}(\Sigma_{g,n})$. 
\label{mapping-class-groups}
\end{corollary}
Corollary \ref{mapping-class-groups} is a generalisation of the following well-known result due to J. Birman, A. Lubotzky and J. McCarthy \cite[Theorem A]{Birman-Lubotzky-McCarthy}: 
if $\mathbb{Z}^m \hookrightarrow \mathcal{M}(\Sigma_{g,n})$, then $m$ does not exceed the maximum order (defined in the next paragraph) of the complete subgraph of $\mathcal{C}(\Sigma_{g,n})$. 
After completing the first draft of this paper,  Koberda
informed the author of the recent paper \cite{Bering-Conant-Gaster} by E. Bering IV, G. Conant, J. Gaster, which gives another combinatorial test for embedding RAAGs into mapping class groups. 
In the final section of this paper, we discuss relation among the result in \cite{Bering-Conant-Gaster}, a result due to Kim-Koberda \cite{Kim-Koberda-3} and Corollary \ref{mapping-class-groups}. 

In Section \ref{plus-minus-technique} of this paper, we refine an embedding  theorem established by Kim-Koberda \cite[Theorem 1.1]{Kim-Koberda-2}, which states that, for any finite graph $\Gamma$, there exists a finite tree $T$ such that $G(\Gamma) \hookrightarrow G(T)$.  
In order to state and explain our refinement, we recall some standard terminology of graph theory. 
For a graph $\Gamma$, the {\it order}, $|\Gamma|$, of $\Gamma$ is the number of the vertices of $\Gamma$.  
The {\it degree} of a vertex $v$ of a graph $\Gamma$, $\mathrm{deg}(v, \Gamma)$, is the number of the edges of $\Gamma$, incident with $v$. 
The {\it maximum degree} of (the vertices in) $\Gamma$ is denoted by $\mathrm{deg}_{\mathrm{max}}(\Gamma)$. 

\begin{theorem}
For each finite graph $\Lambda$, there exists a finite tree $T$ such that $G(\Lambda) \hookrightarrow G(T)$ and $\mathrm{deg}_{\mathrm{max}}(T)\leq 3$. 
\label{Anti-trees-deg3}
\end{theorem}

\begin{remark}
Here, we give some comments on Theorem \ref{Anti-trees-deg3}. 
\begin{enumerate}
 \item[(1)] In the assertion of Theorem \ref{Anti-trees-deg3}, we have $$|T| \leq 
 \begin{cases}
    |\Lambda| \cdot 2^{|\Lambda|} -4 & (\mbox{if }\Lambda \mbox{ is not a tree}) \\
   2|\Lambda| - 4 & (\mbox{if }\Lambda \mbox{ is a tree of maximum degree} >3).  
\end{cases}
$$ 
 \item[(2)] The embedding in the assertion of Theorem \ref{Anti-trees-deg3} can be realised as a quasi-isometric embedding of the Cayley graph of $G(\Lambda)$ into the Cayley graph of $G(T)$. 
 \item[(3)] I. Agol \cite[Theorem 1.1]{Agol} and D. Wise \cite[Theorem 14.29]{Wise} proved that, for any finite volume hyperbolic 3-manifold $M$, there exists a finite graph $\Gamma$ such that the fundamental group $\pi_1(M)$ is virtually embedded into $G(\Gamma)$. 
 Thus, together with Theorem \ref{Anti-trees-deg3}, this implies that for any finite volume hyperbolic 3-manifold $M$, there exists a finite tree $T$ such that $\mathrm{deg}_{\mathrm{max}}(T) \leq 3$ and that $\pi_1(M)$ is virtually embedded into $G(T)$. 
\end{enumerate}
\label{Remark-anti-trees}
\end{remark}

This paper is organised as follows. 
In Section \ref{Preliminaries}, we recall some facts on embeddings between RAAGs. 
We prove Theorems \ref{Main-theorem}(1), \ref{Anti-lines-anti-cycles} and Corollary \ref{mapping-class-groups} in Section \ref{Obstruction}. 
Theorems \ref{Main-theorem}(2) and \ref{Anti-trees-deg3} are proved in Section \ref{plus-minus-technique}. 
Lastly, we discuss in Section \ref{final} a question related to Theorem \ref{Anti-trees-deg3} and relation between Corollary \ref{mapping-class-groups} and recent results in \cite{Bering-Conant-Gaster} and \cite{Kim-Koberda-3}.

\section{Preliminaries \label{Preliminaries}}
In this section we recall some terminology of graph theory and some results on embeddings between RAAGs. 
Let $\Gamma$ be a finite graph. 
\begin{enumerate}
 \item[$\bullet$] The {\it extension graph} $\Gamma^e$ is the graph whose vertex set consists of the words in $A(\Gamma)$ that are conjugate to the vertices of $\Gamma$, and two vertices of $\Gamma^e$ are adjacent if and only if those two vertices commute as words in $A(\Gamma)$. 
 \item[$\bullet$] The {\it link} of the vertex $v$, $\mathrm{Lk}(v, \Gamma)$, is the full subgraph of $\Gamma$ whose vertex set consists of all of the vertices adjacent to $v$. 
Obviously, $|\mathrm{Lk}(v, \Gamma)|= \mathrm{deg}(v, \Gamma)$. 
 \item[$\bullet$] The {\it star} of $v$, $\mathrm{St}(v, \Gamma)$, is the full subgraph of $\Gamma$, whose vertex set consists of $v$ and $V(\mathrm{Lk}(v, \Gamma))$. 
 \item[$\bullet$] The {\it double} of $\Gamma$ along $\mathrm{St}(v, \Gamma)$, $D_{v}(\Gamma)$, is the graph obtained from the disjoint union $\Gamma \sqcup \Gamma'$ by identifying $\mathrm{St}(v, \Gamma)$ and its copy $\mathrm{St}(v', \Gamma') (\leq \Gamma')$, where $\Gamma'$ is a copy of $\Gamma$ and $v'$ is the copy corresponding to $v$ in $\Gamma'$. 
 Obviously, $\Gamma \leq D_{v}(\Gamma)$ and $\mathrm{St}(v, \Gamma) = \mathrm{St}(v', \Gamma') = \mathrm{St}(v, D_{v}(\Gamma))$ hold. 
 Besides, we often denote the double $D_{v}(\Gamma)$ by $\Gamma \cup_{\mathrm{St}(v, \Gamma)} \Gamma'$. 
\end{enumerate}

The following theorems play important roles in the proofs of the main results. 
First of all, the following Theorem \ref{Kim-Koberda} is fundamental for studying embeddings between RAAGs. 

\begin{theorem}[{\cite{Kim-Koberda-1}}]
Let $\Lambda$ and $\Gamma$ be finite graphs. 
\begin{enumerate}
 \item[(1)] We have $A(D_{v}(\Gamma)) \hookrightarrow A(\Gamma)$. 
 \item[(2)] Suppose that $\Lambda$ is a finite full subgraph of the extension graph $\Gamma^e$ of $\Gamma$. 
 Then there exists a finite increasing sequence of full subgraphs of $\Gamma^e$,
 $$\Gamma = \Gamma_0 \leq \Gamma_1 \leq \Gamma_2 \leq \cdots \leq \Gamma_n \leq \Gamma^e ,$$ 
 such that
\begin{enumerate}
 \item[$\bullet$] $\Gamma_i$ is the double of $\Gamma_{i-1}$ along the star of a vertex of $\Gamma_{i-1}$. 
 \item[$\bullet$] $\Lambda \leq \Gamma_n$. 
\end{enumerate}
\end{enumerate}
\label{Kim-Koberda}
\end{theorem}

Casals-Ruiz reduced the embedding problem for certain RAAGs to a graph theoretical problem: 

\begin{theorem}[{\cite[Theorem 3.14]{Casals-Ruiz}}]
Suppose that $\Lambda$ is the complement of a forest and $\Gamma$ is a finite graph. 
Then $A(\Lambda) \hookrightarrow A(\Gamma)$ implies $\Lambda \leq \Gamma^e$. 
\label{Casals-Ruiz}
\end{theorem}

The following is a refinement (due to Lee-Lee) of the embedding theorem of Kim-Koberda \cite[Theorem 3.5]{Kim-Koberda-2}. 

\begin{theorem}[{\cite[Corollary 3.11]{Lee-Lee}}]
Let $\Lambda$ be a finite connected graph and $\widetilde{\Lambda}$ be a universal cover of $\Lambda$. 
Then there exists a finite tree $T \leq \widetilde{\Lambda}$ such that $G(\Lambda) \hookrightarrow A(T)$ and $|T| \leq |\Lambda| \cdot 2^{(|\Lambda| -1)}$. 
\label{Anti-trees-theorem}
\end{theorem}

\section{Proof of Theorem \ref{Main-theorem}(1): an obstruction theorem on embeddings between right-angled Artin groups \label{Obstruction}}

In this section, we prove Theorem \ref{Main-theorem}(1) and some consequences of this result, Theorem \ref{Anti-lines-anti-cycles} and Corollary \ref{mapping-class-groups}. 

We frequently consider a given graph and its complement at the same time. 
The following lemma is obvious from the definitions of the complement of a graph and the double of a graph and the uniqueness of a full subgraph. 

\begin{lemma}
Let $\Lambda$ and $\Gamma$ be finite graphs and $v$ a vertex of $\Gamma$. 
\begin{enumerate}
 \item[(1)] $\Lambda \leq \Gamma$ if and only if $\Lambda^c \leq \Gamma^c$. 
 \item[(2)] In the double $D_v(\Gamma)= \Gamma \cup_{\mathrm{St}(v, \Gamma)} \Gamma'$ of $\Gamma$, the following hold. 
 For each $u \in V(\Gamma) \setminus V(\mathrm{St}(v, \Gamma))$ and $w' \in V(\Gamma') \setminus V(\mathrm{St}(v', \Gamma'))$, $u$ and $w'$ span an edge in the complement of $\Gamma \cup_{\mathrm{St}(v, \Gamma)} \Gamma'$, but do not in the original graph $\Gamma \cup_{\mathrm{St}(v, \Gamma)} \Gamma'$. 
  \item[(3)] Let $V'$ be a subset of $V(\Lambda)$. 
  If $\Lambda \leq \Gamma$, then $\Lambda[V'] = \Gamma[V']$, where $\Lambda[V']$ (resp. $\Gamma[V']$) denotes the full subgraph of $\Lambda$ (resp. $\Gamma$) induced by $V'$. 
\end{enumerate}
\label{Obvious}
\end{lemma}

The following lemma is used in the proof of Theorem \ref{Main-theorem}(1). 
The {\it join} $\Lambda_1 * \Lambda_2 * \cdots * \Lambda_m$ of finite graphs $\Lambda_1, \Lambda_2, \ldots, \Lambda_m$ is the finite graph obtained from the disjoint union $\Lambda_1 \sqcup \Lambda_2 \sqcup \cdots \sqcup \Lambda_m$ by adding all of the edges of the form $[ v_i , v_j ]$ to $\Lambda_1 \sqcup \Lambda_2 \sqcup \cdots \sqcup \Lambda_m$ for all $v_i \in \Lambda_i, \ v_j \in \Lambda_j \ (i \neq j)$. 
We call each $\Lambda_i$ a {\it join-component} of $\Lambda_1 * \Lambda_2 * \cdots * \Lambda_m$. 

\begin{lemma}
Suppose that $\Lambda$ is the complement of a linear forest, $\Gamma$ is a finite graph and $v$ is a vertex of $\Gamma$. 
Then $\Lambda \leq D_{v}(\Gamma)$ implies $\Lambda \leq \Gamma$. 
\label{Extension}
\end{lemma}
\begin{proof}
In this proof, we denote $D_v(\Gamma)= \Gamma \cup_{\mathrm{St}(v, \Gamma)} \Gamma'$ by $D$ for simplicity. 
By the assumption, $\Lambda^c = \sqcup_{i=1}^{m} P_{n_i}$ and so $\Lambda = *_{i=1}^{m} P_{n_i}^c$. 
By setting $\Lambda_i = P_{n_i}^c$, we have $\Lambda = *_{i=1}^{m} \Lambda_i$. 

We now suppose $\Lambda \leq D = \Gamma \cup_{\mathrm{St}(v, \Gamma)} \Gamma'$. 
If $V(\Lambda)$ is contained in either $V(\Gamma)$ or its copy $V(\Gamma')$ in $D$, then we immediately obtain the desired result, so we assume that $V(\Lambda)$ is contained in neither $V(\Gamma)$ nor $V(\Gamma')$. 

\begin{claim}
After changing of indices, we may assume that $$V(\Lambda) \setminus V(\mathrm{St}(v, D)) = V(\Lambda_1) \setminus V(\mathrm{St}(v, D)).$$ 
\label{The-join-component}
\end{claim}
\begin{proof}[Proof of Claim \ref{The-join-component}.] 
By the assumption, there are vertices $u_1, w_1' \in V(\Lambda)$ such that $u_1 \in V(\Gamma)) \setminus V(\mathrm{St}(v, \Gamma))$ and $w_1' \in V(\Gamma') \setminus V(\mathrm{St}(v', \Gamma'))$. 
Then, by Lemma \ref{Obvious}(2), $u_1$ and $w_1'$ do not span an edge in $D$. 
Since $\Lambda$ is a subgraph of $D$, $u_1$ and $w_1'$ do not span an edge in $\Lambda$. 
This shows that $u_1$ and $w_1'$ are contained in the same join-component of $\Lambda$, say $\Lambda_1$. 
Pick any vertex $x$ of $V(\Lambda) \setminus V(\mathrm{St}(v, D))$. 
If $x \in V(\Gamma)$, then by the above argument, we see that $x$ and $w_1'$ belong to the same join-component of $\Lambda$, and so $x \in V(\Lambda_1)$. 
Similarly, if $x \in V(\Gamma')$, then again we have $x \in V(\Lambda_1)$. 
Thus we obtain $V(\Lambda) \setminus V(\mathrm{St}(v, D)) \subset V(\Lambda_1) \setminus V(\mathrm{St}(v, D))$. 
Since the converse inclusion is obvious, we obtain the desired result. 
\end{proof}

Let $\check{\Lambda}_1$ be the full subgraph $\Lambda[ V(\Lambda_1) \setminus V(\mathrm{St}(v, D))]$ of $\Lambda \leq D$. 

\begin{claim}
The full subgraph $\check{\Lambda}_1$ is ismorphic to either $P_2^c$ or $P_3^c$. 
\label{P3}
\end{claim}
\begin{proof}[Proof of Claim \ref{P3}.]
We first show that $|\check{\Lambda}_1| \leq 3$. 
Suppose to the contrary that $|\check{\Lambda}_1| \geq 4$. 
Then, since each of $\Gamma$ and $\Gamma'$ contains a vertex of $\check{\Lambda}_1$, we can find four vertices of $\check{\Lambda}_1$ as in Figure \ref{a}, which imply that $\Lambda_1^c$ has either $C_4$ (as a subgraph) or a vertex of degree $\geq 3$. 
\begin{figure}
\centering
\includegraphics[scale=0.24,clip]{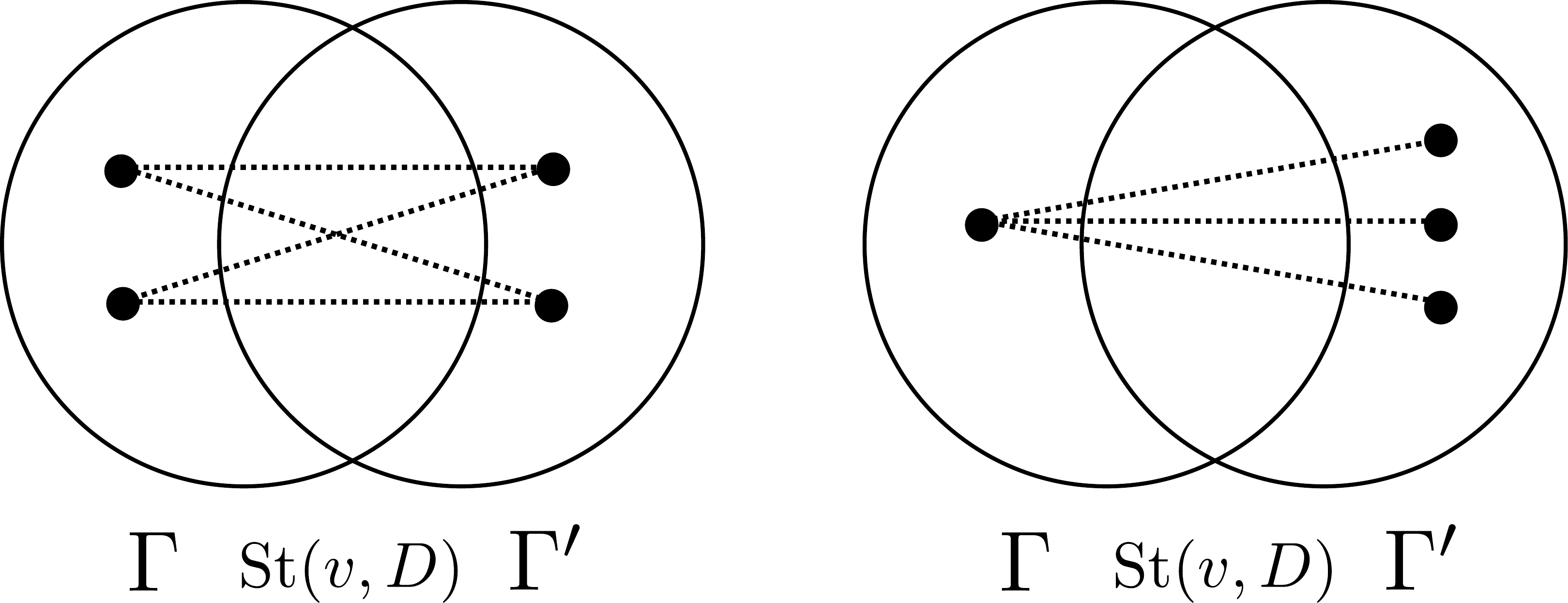}
\caption{
In this schematic figure of $D = \Gamma \cup_{\mathrm{St}(v, D)} \Gamma'$, dotted lines represent edges of the complement $D^c$. 
The left picture illustrates a (possibly non-induced) cycle of length $4$ in the complement of the double $D$. 
The right picture illustrates a vertex of degree $\geq 3$ in the complement of the double $D$.}
\label{a}
\end{figure}
This contradicts to the assumption that $\Lambda_1^c \cong P_{n_1}^c$ is a path graph. 
Thus we have $2 \leq |\check{\Lambda}_1| \leq 3$. 

Suppose first that $|\check{\Lambda}_1| = 2$. 
By Lemma \ref{Obvious}(2), the two vertices of $\check{\Lambda}_1$ do not span an edge in $\check{\Lambda}_1$. 
Hence, $\check{\Lambda}_1 \cong P_2^c$. 

Suppose next that $|\check{\Lambda}_1| = 3$. 
Then, by Lemma \ref{Obvious}(2), $\check{\Lambda}_1$ contains at most one edge. 
Hence, $\check{\Lambda}_1$ is isomorphic to either $P_3^c$ or $K_3^c = C_3^c$. 
However, $\check{\Lambda}_1 \cong C_3^c$ implies $C_3^c \leq \Lambda_1 = P_{n_1}^c$ and so $C_3 \leq P_{n_1}$ by Lemma \ref{Obvious}(1), which is impossible. 
Thus $\check{\Lambda}_1$ must be isomorphic to $P_3^c$. 
\end{proof}

\begin{claim}
The graph $\Lambda$ does not contain the vertex $v$. 
\label{Nocontain}
\end{claim}
\begin{proof}[Proof of Claim \ref{Nocontain}.]
Suppose to the contrary that $\Lambda$ contains $v$. 
Pick vertices $ u, w' \in V(\check{\Lambda}_1)$ so that $u$ and $w'$ satisfy $u \in V(\Gamma) \setminus V(\mathrm{St}(v, \Gamma))$ and $w' \in V(\Gamma') \setminus V(\mathrm{St}(v', \Gamma'))$. 
Then by the definition of the star, $[v, u ]$ and $[v, w' ]$ are edges in $D^c$. 
Moreover, $[ u, w' ]$ is an edge in $D^c$. 
Hence, $v, u, w'$ induces a full subgraph isomorphic to $C_3$ in $D^c$. 
On the other hand, we have $\Lambda^c \leq D^c$ by the assumption that $\Lambda \leq D$ and Lemma \ref{Obvious}(1). 
Thus $C_3$ is a full subgraph of the linear forest $\Lambda^c \cong \sqcup_{i=1}^{m} P_{n_i}$, a contradiction. 
\end{proof}

Before proceeding the proof of Lemma \ref{Extension},  we now summarize the situation. 
By Claim \ref{P3}, almost all part of $\Lambda = *_{i=1}^m \Lambda_i$ is contained in $\mathrm{St}(v, D) = \mathrm{St}(v, \Gamma) = \mathrm{St}(v, \Gamma')$. 
To be more precise, though the small full subgraph $\check{\Lambda}_1$ of $\Lambda_1$ is not contained in $\mathrm{St}(v, \Gamma)$, the remaining vertices of  $V(\Lambda_1)$ and the all of the remaining join-components $\Lambda_i \ (2 \leq i \leq m)$ and so $*_{i=2}^m \Lambda_i$ are contained in $\mathrm{St}(v, \Gamma)$. 

In the remainder of the proof, we find a full subgraph $\Lambda_1'$ of $\Gamma$ or $\Gamma'$ satisfying the following conditions. 
\begin{enumerate}
 \item[(a)] $\Lambda_1'$ is a full subgraph of $\Gamma$ or $\Gamma'$ isomorphic to either $P_{n_1}^c$ or $P_{n_1 + 1}^c$. 
 \item[(b)] $\Lambda_1'$ is ``joinable" with $*_{i=2}^m \Lambda_i$ in $\Gamma$ or $\Gamma'$, 
 and $\tilde{\Lambda}:= \Lambda_1' * (*_{i=2}^m \Lambda_i)$ is a full subgraph of $\Gamma$ or $\Gamma'$. 
\end{enumerate}
Note that the condition (a) is equivalent to the following condition. 
\begin{enumerate}
 \item[(a$'$)] $(\Lambda_1')^c$ is a full subgraph of $(\Gamma)^c$ or $(\Gamma')^c$ isomorphic to either $P_{n_1}$ or $P_{n+1}$. 
\end{enumerate}
If we prove (a$'$) and (b), then we obtain the desired result, because the original graph $\Lambda$ is either isomorphic to $\tilde{\Lambda}$ or a full subgraph of $\tilde{\Lambda}$. 
We divide the proof into two cases according to whether $\check{\Lambda}_1$ is isomorphic to $P_2^c$ or $P_3^c$. 

{\bf Case 1} $\check{\Lambda}_1 \cong P_3^c$. 
We label the vertex set $V(\check{\Lambda}_1)$ by $\{ v_1 , v_2, v_3 \}$ so that $v_1$ and $v_3$ span an edge in $\check{\Lambda}_1 \leq D$. 
Then $[v_1, v_2 ]$ and $[v_2, v_3 ]$ are edges of $(\check{\Lambda}_1)^c$. 
Hence, we may assume that $v_1, v_3 \in V(\Gamma) \setminus V(\mathrm{St}(v, \Gamma))$ and $v_2 \in V(\Gamma') \setminus V(\mathrm{St}(v', \Gamma'))$. 
Let $\Lambda_1' = \Gamma[(V(\Lambda_1) \setminus \{ v_2 \}) \sqcup \{ v \}]$ be the full subgraph of $\Gamma$ induced by $(V(\Lambda_1) \setminus \{ v_2 \}) \sqcup \{ v \}$. 
\begin{figure}
\centering
\includegraphics[scale=0.24,clip]{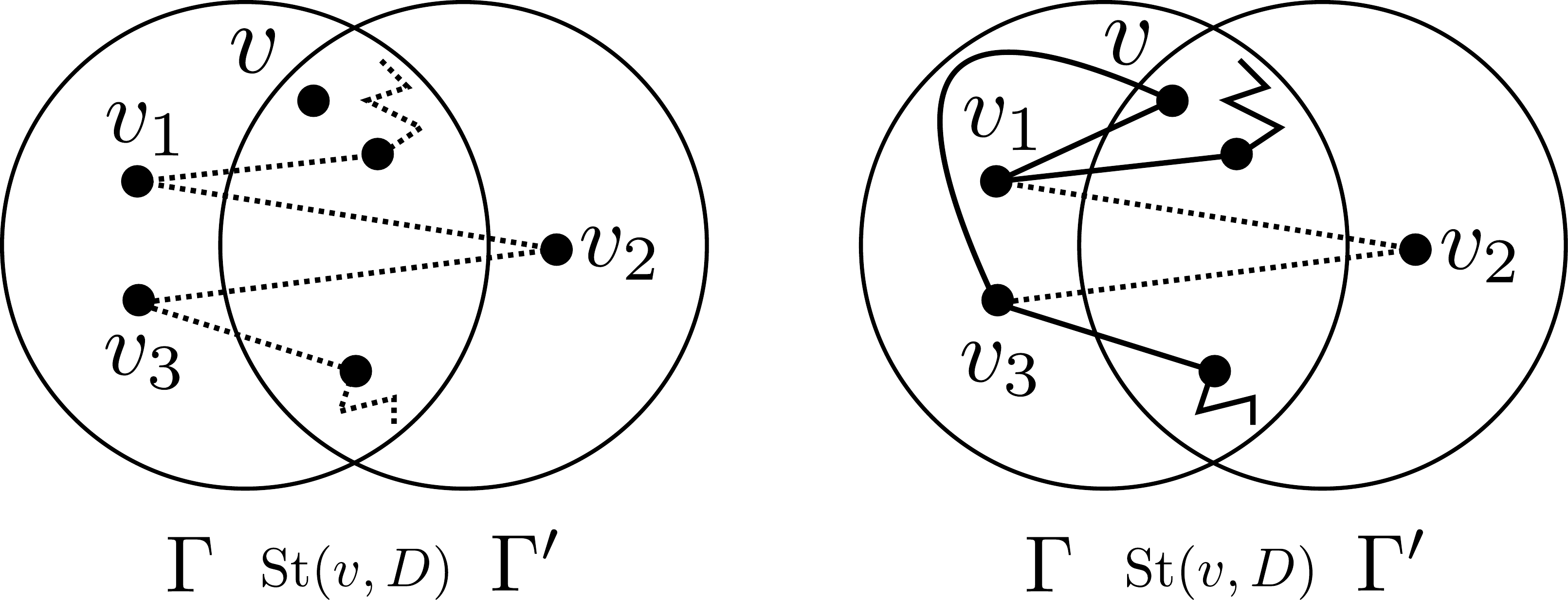}
\caption{Dotted and real lines represent pairs of non adjacent vertices in the double $D$. 
The left picture illustrates $\Lambda_1$ in the double $D$. 
The right picture illustrates replacing $v_2$ with $v$. 
}
\label{g}
\end{figure}
Then, as illustrated in Figure \ref{g}, we can see $(\Lambda_1')^c \cong P_{n_1}$ and hence $\Lambda_1'$ satisfies the condition (a$'$). 
In fact, the map $\phi: V(\Lambda_1) \rightarrow (V(\Lambda_1) \setminus \{ v_2 \}) \sqcup \{ v \}$ define by
$$\phi(x) = 
 \begin{cases}
    v & (\mbox{if} \ x = v_2) \\
    x & (otherwise)
\end{cases}
$$ 
induces an isomorphism from $\Lambda_1^c$ onto $(\Lambda_1')^c$. 
To see this, observe that 
$$V(\Lambda_1) \cap V(\mathrm{St}(v, \Gamma)) = V(\Lambda_1) \setminus \{v_1, v_2, v_3 \}.$$
This implies that the vertex $v$ and a vertex $x \in V(\Lambda_1) \setminus \{ v_2 \}$ span an edge in $\Gamma^c$ if and only if $x \in \{ v_1 , v_3\}$. 
Thus $\phi$ induces an isomorphism. 
Next we show that $\Lambda_1'$ satisfies the condition (b). 
Note that $*_{i=2}^m \Lambda_i$ is a full subgraph of $\mathrm{St}(v, \Gamma)$. 
Therefore, for any $x \in V(*_{i=2}^m \Lambda_i)$, the vertices $v$ and $x$ span an edge in $\Gamma$. 
On the other hand, $\Gamma[V(\Lambda_1) \setminus \{ v_2 \}]*( *_{i=2}^m \Lambda_i) \leq \Gamma$. 
Thus, for any $x_1 \in (V(\Lambda_1) \setminus \{ v_2 \}) \sqcup \{ v \}$ and $x_2 \in V(*_{i=2}^m \Lambda_i)$, the vertices $x_1$ and $x_2$ span an edge in $\Gamma$. 
Since $\Lambda_1' \leq \Gamma$ and $*_{i=2}^m \Lambda_i \leq \Gamma$, wee see $\Lambda_1' * (*_{i=2}^m \Lambda_i) \leq \Gamma$, as desired.  

{\bf Case 2} $\check{\Lambda}_1 \cong P_2^c$. 
We label $V(\check{\Lambda}_1)$ by $\{ v_1' , v_2 \}$ so that $v_1' \in V(\Gamma') \setminus V(\mathrm{St}(v', \Gamma'))$ and $v_2 \in V(\Gamma) \setminus V(\mathrm{St}(v, \Gamma))$. 
Then, $v_1'$ and $v_2$ are not adjacent in $D$. 
We divide the proof into the following two cases (1) and (2). 

{\bf (1)} $\mathrm{deg}(v_1', \Lambda^c) = 1$. 
Set $\Lambda_1' := \Gamma[(V(\Lambda_1) \setminus \{ v_1' \}) \sqcup \{ v \}]$. 
A similar argument as in Case 1 implies that $\Lambda_1'$ satisfies the conditions (a$'$) (in particular we have $(\Lambda_1')^c \cong P_{n_1}$) and (b). 

{\bf (2)} $\mathrm{deg}(v_1', \Lambda^c) = 2$. 
Since $\Gamma = \Gamma'$ and since $v_1' \notin \mathrm{St}(v)$, the set $V(\Gamma) \setminus V(\mathrm{St}(v, \Gamma))$ has the copy, $v_1$, of $v_1'$. 
We pick the vertex $u \in V(\Lambda_1) \cap V(\mathrm{St}(v, D))$ so that $u$ is not adjacent to $v_1'$ in $\Gamma'$ and that $u \neq v_2$ by using the fact that $\mathrm{deg}(v_1', \Lambda^c) = 2$. 
Then the vertex $v_2$ is adjacent to $u$ and $v_1$ is not adjacent to $u$. 
Hence, we have $v_1 \neq v_2$. 
So we furthermore divide the case (2) into the following two cases (2-i) and (2-ii). 
\begin{figure}
\centering
\includegraphics[scale=0.25,clip]{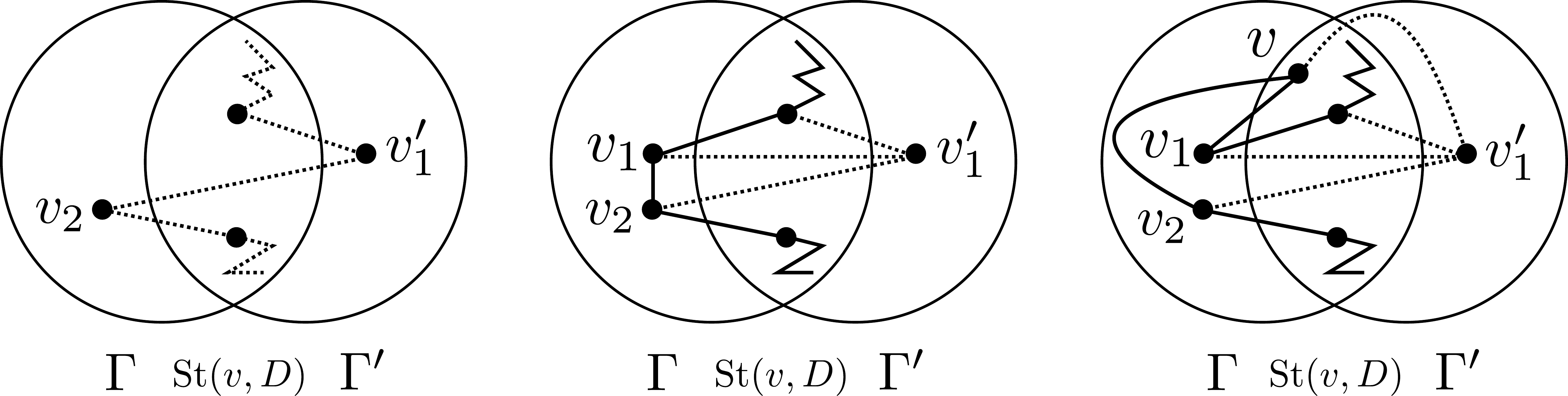}
\caption{The left picture illustrates $\Lambda_1 \leq D = \Gamma \cup_{\mathrm{St}(v, \Gamma)} \Gamma'$. 
Dotted and real lines represent pairs of non-adjacent vertices. 
The picture on the centre illustrate replacing $v_1'$ with $v_1$ in the case (2-1) and the right picture illustrates modifying $P_{n_1}^c$ into $P_{n_1 + 1}^c$ in the case (2-2). 
}
\label{e}
\end{figure}

{\bf (2-i)} The vertices $v_1$ and $v_2$ are not adjacent in $\Gamma$. 
Set $\Lambda_1' := \Gamma[(V(\Lambda_1) \setminus \{ v_1' \}) \sqcup \{ v_1 \}]$. 
Then as illustrated in the centre of Figure \ref{e}, $\Lambda_1'$ satisfies the condition (a$'$) $ (\Lambda_1')^c \cong P_{n_1}$. 
Moreover, we can show that $\Lambda_1'$ satisfies the condition (b) as follows. 
We first see $*_{i=2}^m \Lambda_i \leq \mathrm{St}(v_1, \Gamma)$. 
Pick any vertex $x \in V(*_{i=2}^m \Lambda_i) \subset V(\mathrm{St}(v, \Gamma)) = V(\mathrm{St}(v, \Gamma'))$. 
Then since the vertices $v_1'$ and $x$ span an edge in $\Gamma'$, and since $v_1$ is the copy corresponding to $v_1'$, we see that $v_1$ and $x$ span an edge in $\Gamma$. 
So $*_{i=2}^m \Lambda_i \leq \mathrm{St}(v_1, \Gamma)$ holds. 
Besides, we have $\Gamma[V(\Lambda_1) \setminus \{ v_1' \}] * (*_{i=2}^m \Lambda_i )\leq \Gamma$. 
Hence, for any $x_1 \in V(\Lambda_1')$ and $x_2 \in V(*_{i=2}^m \Lambda_i)$, the vertices $x_1$ and $x_2$ span an edge in $\Gamma$. 
Thus $\Lambda_1'$ satisfies the condition (b). 
 
{\bf (2-ii)} The vertices $v_1$ and $v_2$ are adjacent in $\Gamma$. 
Set $\Lambda_1' := \Gamma[(V(\Lambda_1) \setminus \{ v_1' \}) \sqcup \{ v, v_1 \}]$. 
Then as illustrated in the right of Figure \ref{e}, we have $(\Lambda_1')^c \cong P_{n_1 + 1}$ and so $\Lambda_1'$ satisfies the condition (a$'$). 
It is now a routine work to see the following. 
\begin{enumerate}
 \item[$\bullet$] $*_{i=2}^m \Lambda_i \leq \mathrm{St}(v, \Gamma)$. 
 \item[$\bullet$] $*_{i=2}^m \Lambda_i \leq \mathrm{St}(v_1, \Gamma)$. 
 \item[$\bullet$] $\Gamma[V(\Lambda_1) \setminus \{ v_1' \}] * (*_{i=2}^m \Lambda_i )\leq \Gamma$. 
\end{enumerate}
These show that $\Lambda_1'$ satisfies the condition (b). 
\end{proof}

We prove the following theorem, which is equivalent to Theorem \ref{Main-theorem}(1). 

\begin{theorem}
Suppose that $\Lambda$ is the complement of a linear forest and $\Gamma$ is a finite graph. 
Then $A(\Lambda) \hookrightarrow A(\Gamma)$ implies $\Lambda \leq \Gamma$. 
\label{main-theorem-2}
\end{theorem}
\begin{proof}
Let $\Lambda$ be the complement of a linear forest and $\Gamma$ a finite graph, and suppose that $A(\Lambda) \hookrightarrow A(\Gamma)$. 
Then Theorem \ref{Casals-Ruiz} implies $\Lambda \leq \Gamma^e$. 
By Theorem \ref{Kim-Koberda}(2), there exists a finite increasing sequence of full subgraphs of $\Gamma^e$,
 $$\Gamma = \Gamma_0 \leq \Gamma_1 \leq \Gamma_2 \leq \cdots \leq \Gamma_n \leq \Gamma^e .$$ 
 such that
\begin{enumerate}
 \item[$\bullet$] $\Gamma_i$ is the double of $\Gamma_{i-1}$ along the star of a vertex of $\Gamma_{i-1}$. 
 \item[$\bullet$] $\Lambda \leq \Gamma_n$. 
\end{enumerate}
Hence by repeatedly using Lemma \ref{Extension}, we see $\Lambda \leq \Gamma$, as desired. 
\end{proof}

\begin{proof}[{\bf Proof of Theorem \ref{Main-theorem}(1)}.]
Let $\Lambda$ be a finite linear forest and suppose that $G(\Lambda) \hookrightarrow G(\Gamma)$ for some finite graph $\Gamma$. 
Then $A(\Lambda^c) \hookrightarrow A(\Gamma^c)$ and hence we have $\Lambda^c \leq \Gamma^c$ by Theorem \ref{main-theorem-2}. 
Hence, by Lemma \ref{Obvious}(1), we have $\Lambda \leq \Gamma$. 
\end{proof}

\begin{proof}[{\bf Proof of Theorem \ref{mapping-class-groups}}.]
Suppose that $\phi: A(\Lambda) \hookrightarrow \mathcal{M}(\Sigma_{g,n})$ is an embedding. 
Then as in the proof of \cite[Lemma 2.3]{Kim-Koberda-3}, there exists an induced subgraph $X$ of $\mathcal{C}(\Sigma_{g,n})$ such that $A(\Lambda) \hookrightarrow A(X) \hookrightarrow \mathcal{M}(\Sigma_{g,n})$. 
Thus Theorem \ref{main-theorem-2} implies that $\Lambda \leq X$, and so $\Lambda \leq \mathcal{C}(\Sigma_{g,n})$, as desired. 
\end{proof}

\begin{lemma}
Let $\Lambda$ be a finite graph and $\bar{\Lambda}$ be a subdivision of $\Lambda$. 
Then $G(\Lambda) \hookrightarrow G(\bar{\Lambda})$. 
\label{Complementary-subdivision}
\end{lemma}
\begin{proof}
We may assume that $\bar{\Lambda}$ is obtained from $\Lambda$ by subdividing an edge $[ v, w ]$ of $\Lambda$ into two edges $[v, u]$ and $[u, w]$, where $u$ is a new vertex. 
Note that 
\begin{align*}
V(\bar{\Lambda}) &= V(\Lambda) \sqcup \{ u \} \\
E(\bar{\Lambda}) &= (E(\Lambda) \setminus \{ [v,w] \})   \sqcup \{ [v, u] , [u,w] \} \\
\end{align*}
Then the desired result follows from Claim \ref{Subdiv}. 
\begin{claim}
$\Lambda^c \leq D_u(\bar{\Lambda}^c)$ and hence $A(\Lambda^c) \hookrightarrow A(D_u(\bar{\Lambda}^c))$, where $\bar{\Lambda}^c = (\bar{\Lambda})^c$ is the complement of $\bar{\Lambda}$. 
\label{Subdiv}
\end{claim}
In fact, we have $A(D_u(\bar{\Lambda}^c)) \hookrightarrow A(\bar{\Lambda}^c)$ by Theorem \ref{Kim-Koberda}(1) and therefore 
$$ G(\Lambda) = A(\Lambda^c) \hookrightarrow A(D_u(\bar{\Lambda}^c)) \hookrightarrow A(\bar{\Lambda}^c) = G(\bar{\Lambda}),$$
as desired. 
To prove the above claim, we note that 
$V(\mathrm{St}(u, \bar{\Lambda}))$ consists of the three vertex $u, v$ and $w$. 
Hence, 
$$V(\mathrm{St}(u, \bar{\Lambda}^c)) = \{u \} \sqcup (V(\Lambda^c) \setminus \{v, w \}), $$
and 
$$ V(D_u(\bar{\Lambda}^c)) = \{ u \} \sqcup \{v, w \} \sqcup \{v', w' \} \sqcup (V(\Lambda^c) \setminus \{ v, w \}),$$ 
where $v'$ and $w'$ are the copies of the vertex $v$ and $w$ in $(\bar{\Lambda}^c)'$. 
Let $\Delta$ be the full subgraph of $D_u(\bar{\Lambda}^c)$ induced by $V(\Delta):= (V(\Lambda^c) \setminus \{w \}) \sqcup \{ w' \}$. 
Then there exists a natural bijection $\phi: V(\Lambda^c) \rightarrow V(\Delta)$, whose restriction to $V(\Lambda^c) \setminus \{w \}$ is the identity map (so $\phi$ maps $w$ to $w'$). 
To show that $\phi$ induces an isomorphism from $\Lambda^c$ onto $\Delta$, we prove the following. 
\begin{enumerate}
 \item[(i)] Any two vertices of $V(\Lambda^c) \setminus \{ w \}$ span an edge of $\Lambda^c$ if and only if they span an edge of $\Delta$. 
 \item[(ii)] For each vertex $x$ of $V(\Lambda^c) \setminus \{ w \}$, $[x, w]$ is an edge of $\Lambda^c$ if and only if $[ x, w' ]$ is an edge of $\Delta$. 
\end{enumerate}

We first prove (i). 
To this end, we prove the following identities.      
$$\Lambda^c[V(\Lambda^c) \setminus \{ w \}] = \bar{\Lambda}^c[V(\Lambda^c) \setminus \{ w \}] = (D_u(\bar{\Lambda}^c))[V(\Lambda^c) \setminus \{ w \}] = \Delta[V(\Lambda^c) \setminus \{ w \}] .$$
The first identity follows from the following easy  facts. 
\begin{enumerate}
 \item[$\bullet$] $\Lambda[V(\Lambda) \setminus \{ w \}] = \bar{\Lambda}[V(\Lambda) \setminus \{ w \}]$. 
 \item[$\bullet$] $(\Lambda[V(\Lambda) \setminus \{ w \}])^c  = \Lambda^c[V(\Lambda) \setminus \{ w \}] = \Lambda^c[V(\Lambda^c) \setminus \{ w \}]$. 
 \item[$\bullet$] $(\bar{\Lambda}[V(\Lambda) \setminus \{ w \}])^c = \bar{\Lambda}^c[V(\Lambda^c) \setminus \{ w \}]$. 
\end{enumerate}
The second and third identities follow from the fact that $\bar{\Lambda}^c, \Delta \leq D_u(\bar{\Lambda}^c)$ and Lemma \ref{Obvious}(3). 
Thus we obtain the desired identity $\Lambda^c[V(\Lambda^c) \setminus \{ w \}] = \Delta[V(\Lambda^c) \setminus \{ w \}]$. 
The assertion (i) is immediate from this identity. 

We now prove (ii). 
Pick a vertex $x$ of $V(\Lambda^c) \setminus \{ w \}$.  
If $x = v$, then $[v, w]$ is not an edge in $\Lambda^c$ and $[ v, w' ]$ is not an edge in $D_u(\bar{\Lambda}^c)$ by Lemma \ref{Obvious}(2). 
So we assume $x \neq v$. 
Then the following hold. 
\begin{align*}
 [x, w] \mbox{ is an edge of } \Lambda^c 
& \Leftrightarrow [x, w] \mbox{ is an edge of } \bar{\Lambda}^c \ \ (x \neq v, w) \\
& \Leftrightarrow [x, w] \mbox{ is an edge of } D_u(\bar{\Lambda}^c) \\ 
& \Leftrightarrow [x, w'] \mbox{ is an edge of } D_u(\bar{\Lambda}^c) \ \ (x \in V(\mathrm{St}(u, \bar{\Lambda}^c))) \\
& \Leftrightarrow [x, w'] \mbox{ is an edge of } \Delta
\end{align*}
Thus we obtain the assertion (ii). 

By using the assertion (i) and (ii), we see that $\phi$ induces an isomorphism from $\Lambda^c$ onto $\Delta$. 
So we obtain Claim \ref{Subdiv} and Lemma \ref{Complementary-subdivision}. 
\end{proof}

\begin{proof}[{\bf Proof of Theorem \ref{Anti-lines-anti-cycles}}.] 

(1) Since $P_m$ is a linear forest, $G(P_m) \hookrightarrow G(P_n)$ if and only if $P_m \leq P_n$. 
 The latter relation is equivalent to the inequality $m \leq n$, hence we obtain the desired result. 
 
(2) The proof is similar as in the proof of (1). 

(3) If $m \leq n$, then $C_n$ is a subdivision of $C_m$, and so $G(C_m) \hookrightarrow G(C_n)$ by Lemma \ref{Complementary-subdivision}. 
 Suppose that $G(C_m) \hookrightarrow G(C_n)$. 
To see the converse, note that $P_{m-1} \leq C_m$. 
This implies that $G(P_{m-1}) \hookrightarrow G(C_m)$. 
Hence, $G(P_{m-1}) \hookrightarrow G(C_n)$. 
Thus we have $m-1 \leq n-1$ by Theorem \ref{Anti-lines-anti-cycles}(2), that is, $m \leq n$. 

(4-1) Note that $G(C_3) \cong F_3$ and $G(P_2) \cong F_2$, and $G(P_1) \cong \mathbb{Z}$. 
Thus $G(C_3) \hookrightarrow G(P_2)$ but $G(C_3)$ does not embed into $G(P_1)$. 

(4-2) We first show that $G(C_4) \hookrightarrow G(P_3)$. 
 To this end, note that $C_4^c = P_2 \sqcup  P_2$ and $P_3^c = P_1 \sqcup P_2$. 
 The double $D_u(P_3^c)$, where $u$ is the isolated vertex of $P_3^c$, is isomorphic to $P_1 \sqcup P_2 \sqcup P_2$, and so $C_4^c$ is a full subgraph of $D_u(P_3^c)$. 
 Hence, we have 
$$G(C_4) = A(C_4^c) \hookrightarrow A(D_u(P_3^c)) \hookrightarrow A(P_3^c)= G(P_3), $$
where the second embedding $A(D_u(P_3^c)) \hookrightarrow A(P_3^c)$ follows from Theorem \ref{Kim-Koberda}(1). 
Thus we have only to show that $G(C_4)$ can not be embedded into $G(P_2)$. 
But, this follows from the fact that $G(C_4)$ contains $\mathbb{Z}^2$ whereas $G(P_2) \cong F_2$ does not. 

(4-3) Suppose $G(C_m) \hookrightarrow G(P_n)$. 
Since $G(P_{m-1}) \hookrightarrow G(C_m)$, we have $G(P_{m-1}) \hookrightarrow G(P_n)$, and so $m-1 \leq n$ by Theorem \ref{Anti-lines-anti-cycles}(1). 
\end{proof}

\section{Proofs of Theorems \ref{Main-theorem}(2) and \ref{Anti-trees-deg3} \label{plus-minus-technique}}

\begin{proposition}
Suppose that $\Lambda$ is not a linear forest and that  $\mathrm{deg}_{\mathrm{max}}(\Lambda) \leq 2$. 
Then there exists a finite graph $\Gamma$ such that $G(\Lambda) \hookrightarrow G(\Gamma)$, though $\Lambda \not\leq \Gamma$. 
\label{Anti-deg-2}
\end{proposition} 
\begin{proof}
Our assumptions imply that each connected component of $\Lambda$ is either a path graph or a cyclic graph and $\Lambda$ contains a cyclic graph. 
Hence, $\Lambda$ is the disjoint union of path graphs and cyclic graphs and the RAAG $G(\Lambda)$ is isomorphic to the direct product $G(P_{i_1}) \times \cdots \times G(P_{i_m}) \times G(C_{i_{m+1}}) \times \cdots G(C_{i_{m+n}})$, where $i_1, \ldots , i_{m+n}$ are positive integers and $i_{m+1}, \ldots , i_{m+n}$ are not less than $3$. 
Set $l := 1 + \mathrm{max} \{ i_{j} | 1 \leq j \leq m + n \}$, and consider the graph $\Gamma:= \sqcup^{m+n} C_l$. 
Then $\Lambda$ cannot be embedded into $\Gamma$, but $G(\Lambda) \hookrightarrow G(\Gamma)$ by Theorem \ref{Anti-lines-anti-cycles}(2)(3). 
\end{proof} 

Proposition \ref{Anti-deg-2} proves Theorem \ref{Main-theorem}(2) in the case where $\mathrm{deg}_{\mathrm{max}}(\Lambda) \leq 2$. 
In the following, we treat the case where $\mathrm{deg}_{\mathrm{max}}(\Lambda) \geq 3$. 

\begin{definition}
Let $\Lambda$ be a finite graph and $u$ a vertex of $\Lambda$ with $\mathrm{deg}(u, \Lambda) \geq 2$. 
Pick two vertices $w_1$ and $w_2$ from $\mathrm{Lk}(u, \Lambda)$. 
\begin{enumerate}
 \item[$(-)$] $\Lambda_u^{-}= \Lambda_u^{-}(w_1, w_2)$ denotes the graph with the following property. 
\begin{enumerate}
 \item[$\bullet$] $V(\Lambda_u^{-}) = V(\Lambda) \sqcup \{v \}$, where $v$ is a new vertex. 
 \item[$\bullet$] $E(\Lambda_u^{-}) = (E(\Lambda) \setminus \{ [u, w_1] , [u, w_2] \}) \sqcup \{ [v, w_1], [v, u], [v, w_2] \}$. 
\end{enumerate}
 \item[$(+)$] $\Lambda_u^{+}= \Lambda_u^{+}(w_1, w_2)$ denotes the graph obtained from $\Lambda_u^{-}$ by adding the edge $[u, w_2]$. 
\end{enumerate}
See Figure \ref{b}. 
\begin{figure}
\centering
\includegraphics[scale=0.22,clip]{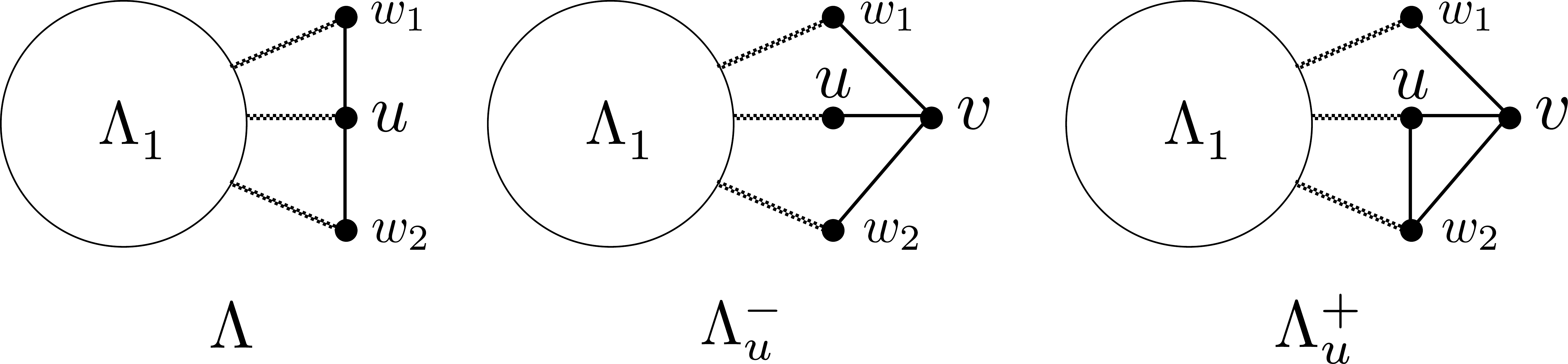}
\caption{
The left picture illustrates $\Lambda$, where $\Lambda_1$ is the full subgraph $\Lambda[V(\Lambda) \setminus \{ u, w_1, w_2 \}]$. 
Mosaic lines represent possible edges while real lines represent edges (in these pictures we omit the possible edge $[w_1, w_2]$ for simplicity). 
The centre picture illustrates $\Lambda_u^{-}(w_1, w_2)$. 
By adding the edge $[ u, w_2 ]$ to the centred  picture, we obtain the right picture which illustrates  $\Lambda_u^{+}(w_1, w_2)$. 
}
\label{b}
\end{figure}
\label{plus-minus}
\end{definition}

We can easily see the following lemma from the definitions of the $(\pm)$-construction. 

\begin{lemma}
Let $\Lambda, \ u, \ w_1, \ w_2$ be as in Definition \ref{plus-minus}. 
\begin{enumerate}
 \item[(1)] $|\Lambda_u^{-}| = |\Lambda| + 1$ and for any vertex $x$ of $\Lambda_u^{-}$, we have 
$$\mathrm{deg} (x, \Lambda_u^{-}) = 
 \begin{cases}
    3 & (\mbox{if} \ x=v) \\
   \mathrm{deg}(x, \Lambda) - 1 & (\mbox{if} \ x=u) \\
   \mathrm{deg}(x, \Lambda) & (otherwise).
\end{cases}
$$ 
Moreover, $\Lambda_u^{-}$ is homotopically equivalent to $\Lambda$. 

 \item[(2)] $|\Lambda_u^{+}| = |\Lambda| + 1$ and for any vertex $x$ of $\Lambda_u^{+}$, we have 
$$\mathrm{deg} (x, \Lambda_u^{+}) = 
 \begin{cases}
    3 & (\mbox{if} \ x=v) \\
   \mathrm{deg}(x, \Lambda) + 1 & (\mbox{if} \ x = w_2) \\
   \mathrm{deg}(x, \Lambda) & (otherwise).
\end{cases}
$$ 
\end{enumerate}
\label{Computation_anti_degree}
\end{lemma}
\begin{proof}
To see that $\Lambda_u^{-}$ is homotopically equivalent to $\Lambda$, contract the edge $[u, v]$ of $\Lambda_u^{-}$ to a single vertex $u$. 
\end{proof}

\begin{lemma}
Let $\Lambda, \ u, \ w_1, \ w_2$ be as in Definition \ref{plus-minus}. 
Then $G(\Lambda) \hookrightarrow G(\Lambda_u^{\epsilon})$ for each $\epsilon= + , -$. 
\label{Heteromorphic_lemma}
\end{lemma}
\begin{proof}
For simplicity, we assume $\epsilon = -$ (the case $\epsilon = +$ can be treated similarly). 
Let $\Lambda_u^{-c} := (\Lambda_u^{-})^c$ be the complement of $\Lambda_u^{-}$, and consider the double $D_v(\Lambda_u^{-c})$ of $\Lambda_u^{-c}$ along $\mathrm{St}(v, \Lambda_u^{-c})$, where $v$ is the vertex of $\Lambda_u^{-c}$ corresponding to the new vertex $v$ of $\Lambda_u^{-}$. 
We show that the complement $\Lambda^c$ of $\Lambda$ can be embedded into $D_v(\Lambda_u^{-c})$ as a full subgraph, and hence $A(\Lambda^c) \hookrightarrow A(D_v(\Lambda_u^{-c}))$. 
Then we obtain the desired embedding $G(\Lambda) = A(\Lambda^c) \hookrightarrow A(\Lambda_u^{-c}) = G(\Lambda_u^{-})$ by Theorem \ref{Kim-Koberda}(1). 
To construct an embedding of $\Lambda^c$ into $D_v(\Lambda_u^{-c})$, observe that 
$$V(\mathrm{St}(v, \Lambda_u^{-c})) = \{ v \} \sqcup V(\Lambda_1) .$$
Let $\Delta$ be the full subgraph of $D_v(\Lambda_u^{-c})$, induced by $V(\Delta):= (V(\Lambda^c) \setminus \{ u \}) \sqcup \{ u' \}$, where $u'$ is the copy of $u$ in the copy of $\Lambda_u^{-c}$, $(\Lambda_u^{-c})'$. 
Since $V(\Lambda^c) = (V(\Lambda^c) \setminus \{ u \}) \sqcup \{ u \}$, there exists a natural bijection $\phi: V(\Lambda^c) \rightarrow V(\Delta)$. 
The following (i) and (ii) imply that $\phi$ induces an isomorphism from $\Lambda^c$ onto $\Delta$, completing the proof. 
\begin{enumerate}
 \item[(i)] The restriction of $\phi$ to $\Lambda^c[V(\Lambda^c) \setminus \{ u \}]$ is an isomorphism. 
 \item[(ii)] For each $x \in V(\Lambda^c) \setminus \{ u \}$, $[x, u]$ is an edge of $\Lambda^c$ if and only if $[x, u']$ is an edge of $\Delta$. 
\end{enumerate}
The assertions (i) and (ii) can be proved by an argument similar to the proof of Claim \ref{Subdiv} (see Figure \ref{h}). 
\begin{figure}
\centering
\includegraphics[scale=0.22,clip]{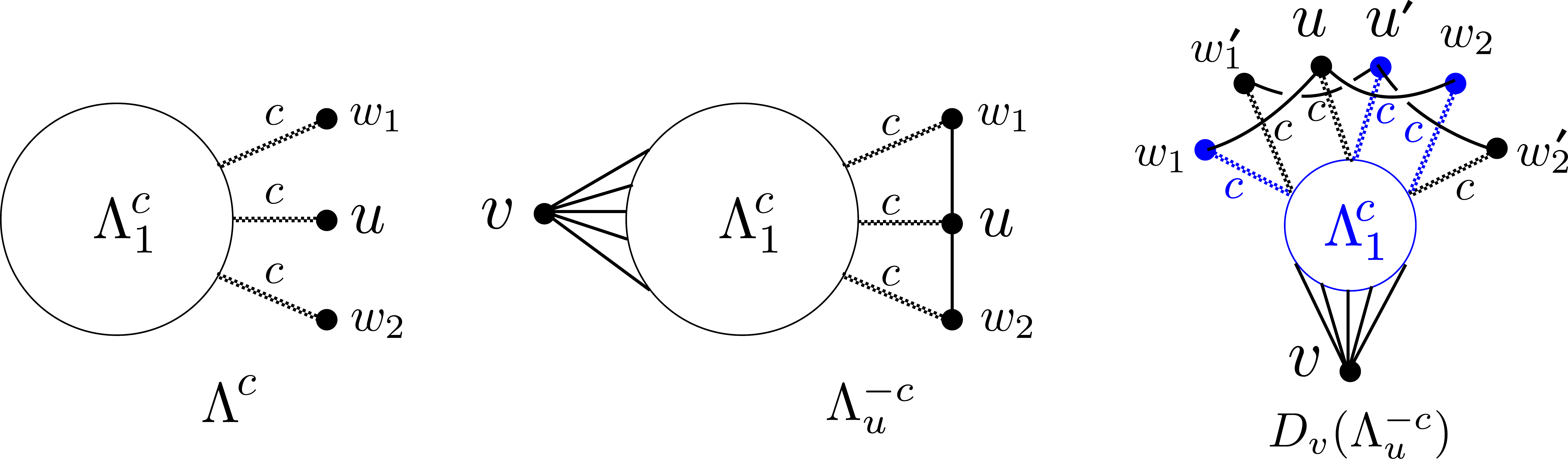}
\caption{
Mosaic lines with the characters $c$ represent possible edges in the complement graphs. 
The left picture illustrates the complement $\Lambda^c$ of $\Lambda$. 
The centred picture illustrates the complement $\Lambda_u^{-c}$ of $\Lambda_u^{-}$. 
The right picture illustrates the double $D_v(\Lambda_u^{-c})$ of $\Lambda_u^{-c}$, along the star $\mathrm{St}(v, \Lambda_u^{-c}) = v * {\Lambda_1^c}$. 
The domain, $D_v(\Lambda_u^{-c})[V(\Lambda_1^c) \sqcup \{ w_1, u', w_2 \}] = D_v(\Lambda_u^{-c})[(V(\Lambda^c) \setminus \{ u \})  \sqcup \{ u' \}]$, illustrates $\Delta$.  
}
\label{h}
\end{figure}
\end{proof}

For a finite graph $\Gamma$ and a natural number $n$, set $V_n(\Gamma): = \{ v \in V(\Gamma) | \ \mathrm{deg} (v, \Gamma) = n \}$.  

\begin{proposition}
If the inequality $\mathrm{deg}_{\mathrm{max}}(\Lambda) \geq 4$ holds, then there exists a finite graph $\Gamma$ such that $G(\Lambda) \hookrightarrow G(\Gamma)$, though $\Lambda \not\leq \Gamma$. 
\label{Anti-deg-4}
\end{proposition}
\begin{proof}
Pick a vertex $u$ of $\Lambda$ with $\mathrm{deg}(u, \Lambda) = \mathrm{deg}_{\mathrm{max}}(\Lambda)\geq 4$ and two vertices $w_1, w_2$ from $\mathrm{Lk}(u, \Lambda)$. 
Then $G(\Lambda) \hookrightarrow G(\Lambda_u^{-})$ by Lemma \ref{Heteromorphic_lemma}, where $\Lambda_u^{-} = \Lambda_u^{-}(w_1, w_2)$. 
However, we can see that $\Lambda \not\leq \Lambda_u^{-}$ as follows. 
Suppose on the contrary that $\Lambda \leq \Lambda_u^{-}$. 
Then there exists a vertex subset $V' \subset V(\Lambda_u^{-})$ such that $\Lambda_u^{-}[V'] \cong \Lambda$. 
Let $n \geq 4$ be the degree of $u$ in $\Lambda$. 
Then, Lemma \ref{Computation_anti_degree}(1) implies $|V_n(\Lambda_u^{-})| = |V_n(\Lambda)| - 1$. 
We now claim that $V_n (\Lambda_u^{-}[V']) \subset V_n(\Lambda_u^{-})$. 
To see this, pick a vertex $x$ of $V_n (\Lambda_u^{-}[V'])$. 
Then Lemma \ref{Computation_anti_degree}(1) implies  
$$n = \mathrm{deg}(x, \Lambda_u^{-}[V']) \leq \mathrm{deg}(x, \Lambda_u^{-}) \leq \mathrm{deg}_{\mathrm{max}}(\Lambda_u^{-}) \leq \mathrm{deg}_{\mathrm{max}}(\Lambda) = n ,$$
 and so $x \in V_n(\Lambda_u^{-})$. 
Therefore, we obtain $V_n (\Lambda_u^{-}[V']) \subset V_n(\Lambda_u^{-})$ and so 
$$ |V_n(\Lambda_u^{-}[V'])| \leq |V_n(\Lambda_u^{-})| = |V_n(\Lambda)| - 1 .$$ 
This contradicts $\Lambda_u^{-}[V'] \cong \Lambda$. 
Thus we have $\Lambda \not\leq \Lambda_u^{-}$ . 
\end{proof}

\begin{proposition}
Suppose $\mathrm{deg}_{\mathrm{max}}(\Lambda) = 3$ and that for each vertex $x \in V_3(\Lambda)$, $ \mathrm{Lk}(x, \Lambda)$ does not contain an edge, i.e., $\mathrm{Lk}(x, \Lambda) \cong K_3^c$. 
Then there exists a finite graph $\Gamma$ such that $G(\Lambda) \hookrightarrow G(\Gamma)$, though $\Lambda \not\leq \Gamma$. 
\label{Anti-deg-3-complete}
\end{proposition}
\begin{proof}
Pick a vertex $u$ of $\Lambda$, of degree $3$ and two vertices $w_1, w_2$ of $\mathrm{Lk}(u, \Lambda)$ and consider $\Lambda_u^{+} = \Lambda_u^{+}(w_1, w_2)$. 
By Lemma \ref{Heteromorphic_lemma}, we have $G(\Lambda) \hookrightarrow G(\Lambda_u^{+})$. 
So we have only to show $\Lambda \not\leq \Lambda_u^{+}$. 
Suppose, to the contrary, that $\Lambda \leq \Lambda_u^{+}$, namely, there exists a subset $V' \subset V(\Lambda_u^{+})$ such that $\Lambda_u^{+}[V'] \cong \Lambda$. 
Since $|\Lambda| = |\Lambda_u^{+}| - 1$, there exists a vertex $v' \in V(\Lambda_u^{+})$ such that $V' = V(\Lambda_u^{+}) \setminus \{ v' \}$.  

{\bf Case 1.} Suppose that we can choose the above $u, w_1, w_2$ so that $\mathrm{deg}(w_2, \Lambda) \leq 2$. 
To treat this case we prepare a notation. 
For a graph $\Gamma$, set $V_3^e(\Gamma):= \{ v \in V_3(\Gamma) | \ \mathrm{Lk}(v, \Gamma) = K_3^c \}$. 
Note that the assumption of Proposition \ref{Anti-deg-3-complete} implies $V_3^e(\Lambda) = V_3 (\Lambda)$. 

\begin{claim}
The following hold. 
\begin{enumerate}
\item[(i)] $V_3^e (\Lambda_u^{+}) \subset V_3^e (\Lambda) \setminus \{ u \} = V_3 (\Lambda) \setminus \{ u \}$. 
\item[(ii)] $V_3^e(\Lambda_u^{+}[V']) \subset  V_3^e(\Lambda_u^{+})$. 
\end{enumerate}
\label{empty_number}
\end{claim}
\begin{proof}[Proof of Claim \ref{empty_number}.] 
(i) Pick a vertex $x$ of $V_3^e (\Lambda_u^{+})$. 
It obviously follows from the definition of $\Lambda_u^{+}$ (cf. Figure \ref{b}) that $u, v, w_2 \notin V_3^e(\Lambda_u^{+})$, and so $x \in V(\Lambda_u^{+}) \setminus \{ u, v, w_2 \}$.  
Then by Lemma \ref{Computation_anti_degree}(2), we have $\mathrm{deg}(x, \Lambda) = \mathrm{deg}(x, \Lambda_u^{+}) = 3$, which implies that $x \in V_3(\Lambda)$, and so $x \in V_3(\Lambda) \setminus \{ u \}$.  

(ii) Pick a vertex $x$ of $V_3^e(\Lambda_u^{+}[V'])$. 
If $x \in V(\mathrm{Lk}(v', \Lambda_u^{+}))$, where $v'$ is the removed vertex, then $\mathrm{deg}(x, \Lambda_u^{+}) = 4$. 
However, by Lemma \ref{Computation_anti_degree}(2) and the assumption that $\mathrm{deg}(w_2, \Lambda) \leq 2$ and $\mathrm{deg}_{\mathrm{max}}(\Lambda) \leq 3$, there is no vertex of degree $4$ in $\Lambda_u^{+}$, a contradiction. 
Hence, $x \notin V(\mathrm{Lk}(v', \Lambda_u^{+}))$. 
In this case, we have $\mathrm{Lk}(x, \Lambda_u^{+}) = \mathrm{Lk}(x, \Lambda_u^{+}[V']) \cong K_3^c$, and hence $x \in V_3^e(\Lambda_u^{+})$. 
\end{proof}

By the above claim, $|V_3^e(\Lambda_u^{+}[V'])| \leq |V_3^e(\Lambda_u^{+})| \leq |V_3^e (\Lambda)| -1$. 
This contradicts the assumption that $\Lambda_u^{+}[V'] \cong \Lambda$.  

{\bf Case 2.} Suppose that for each $x \in V_3(\Lambda)$, every vertex of $\mathrm{Lk}(x, \Lambda)$ has degree $3$, i.e., 
\begin{enumerate}
 \item[(A)] for each $x \in V_3(\Lambda)$,  $V(\mathrm{Lk}(x, \Lambda)) \subset V_3(\Lambda)$ holds.  
\end{enumerate}

\begin{claim}
The following fold. 
\begin{enumerate}
 \item[(i)] $|V_3 (\Lambda_u^{+})| = |V_3 (\Lambda)|$. 
 \item[(ii)] $V_4(\Lambda_u^{+}) = \{ w_2 \}$ and  $V(\mathrm{Lk}(w_2, \Lambda_u^{+})) \subset V_3(\Lambda_u^{+})$. 
 \item[(iii)] For each vertex $x$ of $\mathrm{Lk}(w_2,  \Lambda_u^{+})$, the set $V(\mathrm{Lk}(x, \Lambda_u^{+}))$ consists of $w_2$ and two vertices of degree $3$ in $\Lambda_u^{+}$. 
 \label{Around-deg3}
 \end{enumerate}
\end{claim}
\begin{proof}[Proof of Claim \ref{Around-deg3}.]
(i) By Lemma \ref{Computation_anti_degree}(2) and the assumptions that $\mathrm{deg}(w_2, \Lambda)=3$, we see $v \in V_3(\Lambda_u^{+})$, $ w_2 \not\in V_3(\Lambda_u^{+})$, and $V_3(\Lambda) \setminus \{ v, w_2 \} = V_3(\Lambda_u^{+}) \setminus \{ v, w_2 \}$. 
Hence, we have $|V_3(\Lambda_u^{+})| = |V_3(\Lambda)|$. 

(ii) By Lemma \ref{Computation_anti_degree}(2) and the assumption that $\mathrm{deg}_{\mathrm{max}}(\Lambda) = 3$, we see $V_4(\Lambda_u^{+}) = \{ w_2 \}$. 
To prove $V(\mathrm{Lk}(w_2, \Lambda_u^{+})) \subset V_3(\Lambda_u^{+})$, pick a vertex $x$ of $\mathrm{Lk}(w_2, \Lambda_u^{+})$. 
If $x = v$, then $\mathrm{deg}(x, \Lambda_u^{+}) = 3$ by Lemma \ref{Computation_anti_degree}(2) and so $x \in V_3(\Lambda_u^{+})$. 
If $x \neq v$, then by Lemma \ref{Computation_anti_degree}(2) again, we have $\mathrm{deg}(x, \Lambda_u^{+}) = \mathrm{deg}(x, \Lambda)$, which is equal to $3$ by the assumption (A) and the fact that $x$ and $w_2$ are adjacent in $\Lambda$ (because $V(\mathrm{Lk}(w_2, \Lambda_u^{+})) = V(\mathrm{Lk}(w_2, \Lambda)) \sqcup \{ v \}$). 
 Hence, we have $x \in V_3(\Lambda_u^{+})$. 

(iii) Let $x$ be a vertex of $\mathrm{Lk}(w_2, \Lambda_u^{+})$. 
Then by (ii), $V(\mathrm{Lk}(x, \Lambda_u^{+}))$ consists of three vertices, one of which is $w_2$. 
Moreover, by the assumption (A) and Lemma \ref{Computation_anti_degree}(2), we can prove that each vertex of $\mathrm{Lk}(x, \Lambda_u^{+})$ different from $w_2$ has degree $3$ in $\Lambda_u^{+}$ as follows. 
Pick a vertex $y$ in $\mathrm{Lk}(x, \Lambda_u^{+})$ different from $w_2$. 
If $y$ is either $u$ or $v$, then $\mathrm{deg}(y, \Lambda_u^{+})= 3$. 
If $y = w_1$, then $\mathrm{deg}(y, \Lambda_u^{+})= \mathrm{deg}(y, \Lambda)$ by Lemma \ref{Computation_anti_degree}(2). 
This is equal to $3$ by the assumption (A) together with the fact that $u$ and $w_1$ are adjacent in $\Lambda$. 
Suppose $y \not\in \{u, v, w_1, w_2 \}$. 
Then $y$ is a vertex of the full subgraph $\Lambda_1$ in Figure \ref{b}, and so $x \neq v$. 
This implies that the edge $[x, y]$ in $\Lambda_u^{+}$ is actually an edge in $\Lambda$. 
Hence $y \in V(\mathrm{Lk}(x, \Lambda))$ and so $\mathrm{deg}(y, \Lambda) = 3$ by the assumption (A).  
Thus, by Lemma \ref{Computation_anti_degree}(2), $\mathrm{deg}(y, \Lambda_u^{+}) = \mathrm{deg}(y, \Lambda) = 3$. 
\end{proof}

Since $\mathrm{deg}(w_2, \Lambda_u^{+}) = 4$ (Claim \ref{Around-deg3}(ii)) and since $\mathrm{deg}_{\mathrm{max}}(\Lambda) = 3$ (the assumption of Proposition \ref{Anti-deg-3-complete}), the removed vertex $v'$ must be contained in $V(\mathrm{St}(w_2, \Lambda_u^{+})) = \{ w_2 \} \sqcup V(\mathrm{Lk}(w_2, \Lambda_u^{+}))$. 

Suppose that $v' = w_2$. 
Then for any $x \in V(\mathrm{Lk}(w_2, \Lambda_u^{+}))$, we have 
$$V(\mathrm{Lk}(x, \Lambda_u^{+}[V'])) \subsetneq V(\mathrm{Lk}(x, \Lambda_u^{+})),$$ 
and so $x \notin V_3(\Lambda_u^{+}[V'])$. 
On the other hand, $V(\mathrm{Lk}(w_2, \Lambda_u^{+})) \subset V_3(\Lambda_u^{+})$, by Claim \ref{Around-deg3}(ii). 
Hence, we have
\begin{align*}
|V_3(\Lambda_u^{+}[V'])| & \leq |V_3(\Lambda_u^{+})| - |V(\mathrm{Lk}(w_2, \Lambda_u^{+}))| \\
 & = |V_3(\Lambda_u^{+})| - 4 \\
 & = |V_3(\Lambda)| - 4 \ \ (\mbox{by Claim \ref{Around-deg3}(i)})
\end{align*}
This contradicts the assumption $\Lambda \cong \Lambda_u^{+}[V']$. 

Suppose $v' \in  V(\mathrm{Lk}(w_2, \Lambda_u^{+}))$. 
By Claim \ref{Around-deg3}(iii), we have 
$V(\mathrm{Lk}(v', \Lambda_u^{+})) = \{ w_2, x_1, x_2 \},$
where $x_1, x_2$ are elements of $V_3(\Lambda_u^{+})$. 
By Claim \ref{Around-deg3}(ii), we see 
$$ w_2 \notin V_3(\Lambda_u^{+}) \ \mbox{but} \ w_2 \in V_3(\Lambda_u^{+}[V']). $$
We can also see that $x_1, x_2 \notin V_3(\Lambda_u^{+}[V'])$. 
Moreover, since removing $v'$ does not decrease the degrees of the vertices of $V_3(\Lambda_u^{+}) \setminus \{ v', w_2, x_1, x_2 \}$, we can see   $V_3(\Lambda_u^{+}[V']) \setminus \{ v', w_2, x_1, x_2 \} = V_3(\Lambda_u^{+})  \setminus \{ v', w_2, x_1, x_2 \}$. 
Hence, we obtain 
$$ |V_3(\Lambda_u^{+}[V'])| = |V_3(\Lambda_u^{+})| +1 -3 ,$$ 
which is in turn equal to $|V_3(\Lambda)| - 2$ by Claim \ref{Around-deg3}(i). 
This again contradicts the assumption that $\Lambda_u^{+}[V'] \cong \Lambda$. 
\end{proof}

\begin{proposition}
Suppose $\mathrm{deg}_{\mathrm{max}}(\Lambda) = 3$ and that there exists a vertex $u \in V_3(\Lambda)$ such that $\mathrm{Lk}(u, \Lambda)$ contains an edge, then there exists a finite graph $\Gamma$ such that $G(\Lambda) \hookrightarrow G(\Gamma)$, though $\Lambda \not\leq \Gamma$. 
\label{Anti-deg-3-incomplete}
\end{proposition}
\begin{proof}
Pick a vertex $u \in V_3(\Lambda)$ such that the link $\mathrm{Lk}(u, \Lambda)$ contains an edge, and two vertices $w_1, w_2$ of $\mathrm{Lk}(u, \Lambda)$ with $[w_1, w_2] \in E(\Lambda)$. 
Consider $\Lambda_u^{-} = \Lambda_u^{-}(w_1, w_2)$. 
Then we have $G(\Lambda) \hookrightarrow G(\Lambda_u^{-})$ by Lemma \ref{Heteromorphic_lemma}. 
So we have only to prove $\Lambda \not\leq \Lambda_u^{-}$. 
To this end, we first see the following. 
\begin{claim} 
The following hold. 
\begin{enumerate}
 \item[(i)] $| \mathrm{Lk}(u, \Lambda_u^{-})| = 2$. 
 \item[(ii)] $E(\mathrm{Lk}(v, \Lambda_u^{-})) = \{ [w_1, w_2] \} $.  
 \item[(iii)] For each element $x$ of  $V_3(\Lambda_u^{-}) \setminus \{ u, v \}$, we have $ x \in V_3(\Lambda)$ and 
 $$ \ |E(\mathrm{Lk}(x, \Lambda_u^{-}))| \leq |E(\mathrm{Lk}(x, \Lambda))| .$$  
\label{the-number-edges}
\end{enumerate}
\end{claim}
\begin{proof}[Proof of Claim \ref{the-number-edges}.]
(i) Use the assumption that $\mathrm{deg}(u, \Lambda) = 3$ and Lemma \ref{Computation_anti_degree}(1). 

(ii) Note that $\mathrm{Lk}(v, \Lambda_u^{-}) = \{ u, w_1, w_2 \}$. 
By the assumption, $[w_1, w_2]$ is an edge of $\Lambda$, and so this is an edge of $\Lambda_u^{-}$. 
However, $[u, w_1], [u, w_2]$ are not edges in $\Lambda_u^{-}$. 

(iii) Let $x$ be an element of $V_3(\Lambda_u^{-}) \setminus \{ u, v \}$. 
Then by Lemma \ref{Computation_anti_degree}(1), we have $\mathrm{deg}(x, \Lambda) = 3$. 
Suppose $x = w_1$. 
Then by the assumption that $[w_1, w_2]$ is an edge in $\Lambda$, we have $V(\mathrm{Lk}(x, \Lambda)) = \{ u, w_2 , y \}$ and $V(\mathrm{Lk}(x, \Lambda_u^{-})) = \{ v, w_2, y \}$, where $y$ is a vertex of $\Lambda_1 = \Lambda [V(\Lambda) \setminus \{ w_1, u , w_2 \}]$ (see Figure \ref{b}). 
Note that $[u, w_2]$ is an edge and $[u, y], [w_2, y]$ are possible edges in the link $\mathrm{Lk}(x, \Lambda)$. 
On the other hand, $[v, w_2]$ is an edge and only $[w_2, y]$ is a possible edge ($v$ and $y$ do not span an edge) in the link $\mathrm{Lk}(x, \Lambda_u^{-})$. 
Hence, in case $x = w_1$, (iii) holds. 
The case $x = w_2$ can be treated similarly. 
Suppose that $x$ is contained in $V(\Lambda_1) = V(\Lambda_u^{-}) \setminus \{w_1, u, v, w_2 \}$. 
Then since $x$ and $v$ are not adjacent, $\mathrm{Lk}(x, \Lambda_u^{-})$ is contained in the full subgraph $\Lambda_u^{-}[V(\Lambda_u^{-}) \setminus \{ v \}]$. 
However, since $\Lambda_u^{-}[V(\Lambda_u^{-}) \setminus \{ v \}]$ is obtained from $\Lambda$ by removing the two edges $[u, w_1]$ and $[u, w_2]$, the number of the edges of $\mathrm{Lk}(x, \Lambda_u^{-})$ is not more than that of $\mathrm{Lk}(x, \Lambda)$. 
Thus, in any case, (iii) holds. 
\end{proof}

Now suppose to the contrary that $\Lambda \leq \Lambda_u^{-}$. 
Then there exists a subset $V' \subset V(\Lambda_u^{-})$ such that $\Lambda_u^{-}[V'] \cong \Lambda$. 

{\bf Case 1.} The link $\mathrm{Lk}(u, \Lambda)$ is complete. 
To treat this case we prepare a notation. 
For a graph $\Gamma$, set $V_3^k(\Gamma):= \{ v \in V_3(\Gamma) | \ \mathrm{Lk}(v, \Gamma) \cong K_3 \}$. 
We first show $|V_3^k(\Lambda_u^{-})| \leq |V_3^k(\Lambda)| - 1$ to obtain a contradiction.  
Note that $V(\Lambda_u^{-}) = (V(\Lambda) \setminus \{ u \}) \sqcup \{ u, v \}$. 
By Claim \ref{the-number-edges}(i) and (ii), we have $u, v \notin V_3^k(\Lambda_u^{-})$. 
Hence, $V_3^k(\Lambda_u^{-}) = V_3^k(\Lambda_u^{-}) \setminus \{ u, v\}$. 
On the other hand, Claim \ref{the-number-edges}(iii) implies $V_3^k(\Lambda_u^{-}) \setminus \{ u, v \} \subset V_3^k(\Lambda)$. 
Thus we obtain $V_3^k(\Lambda_u^{-}) \subset V_3^k(\Lambda) \setminus  \{ u \}$, where $u \in V_3^k(\Lambda)$. 
So we have 
\begin{align*}
|V_3^k(\Lambda_u^{-})| & \leq |V_3^k(\Lambda) \setminus  \{ u \}| \\
 & = |V_3^k(\Lambda)| - 1. 
\end{align*}
Moreover, since $\mathrm{deg}_{\mathrm{max}}(\Lambda_u^{-}) = 3$ and $\Lambda_u^{-}[V']$ is a proper subgraph of $\Lambda_u^{-}$, we have $|V_3^k(\Lambda_u^{-}[V'])| \leq |V_3^k(\Lambda_u^{-})|$ (cf.~ the proof of Claim \ref{empty_number}(ii)). 
Thus $|V_3^k(\Lambda_u^{-}[V'])| \leq |V_3^k (\Lambda_u^{-})| \leq |V_3^k (\Lambda)| - 1$, and this  contradicts $\Lambda_u^{-}[V'] \cong \Lambda$.

{\bf Case 2.} The link $\mathrm{Lk}(u, \Lambda)$ is not complete.
Then, we may assume that $w_1$ and $w_2$ do not span an edge in $\Lambda$, so they do not in $\Lambda_u^{-}$. 
Then we can see that 
\begin{enumerate}
 \item[(ii$'$)] the link $\mathrm{Lk}(v, \Lambda_u^{-})$ contains no edges. 
\end{enumerate}
Now, for a finite graph $\Gamma$, set $V_3^{*}(\Gamma):= \{ v \in V_3(\Gamma) | \ \mathrm{Lk}(v, \Gamma) \not\cong K_3^c \}$. 
Observe that the following hold by Claim \ref{the-number-edges}(i), (iii) and the above (ii$'$). 
\begin{enumerate}
 \item[$\bullet$] $V_3^{*}(\Lambda_u^{-}) \subset V_3^{*}(\Lambda) \setminus \{ u \}$ and $u \in V_3^{*}(\Lambda)$. 
 \item[$\bullet$] $|V_3^{*}(\Lambda_u^{-}[V'])| \leq |V_3^{*}(\Lambda_u^{-})|$. 
\end{enumerate}
Hence, we have $|V_3^{*}(\Lambda_u^{-}[V'])| \leq |V_3^{*}(\Lambda_u^{-})| \leq |V_3^{*}(\Lambda)| - 1$, which contradicts $\Lambda_u^{-}[V'] \cong \Lambda$. 

Thus in both case $\Lambda \not\leq \Lambda_u^{-}$ holds. 
\end{proof}

\begin{proof}[{\bf Proof of Theorem \ref{Main-theorem}(2)}.]
Propositions \ref{Anti-deg-2}, \ref{Anti-deg-4}, \ref{Anti-deg-3-complete} and \ref{Anti-deg-3-incomplete}. 
\end{proof}

Next, we prove Theorem \ref{Anti-trees-deg3}. 

\begin{proposition}
Let $T$ be a finite tree such that $\mathrm{deg}_{\mathrm{max}}(T) \geq 4$. 
Then there exists a finite tree $T'$ such that $G(T) \hookrightarrow G(T')$ and that $\mathrm{deg}_{\mathrm{max}}(T') \leq 3$ and $|T'| \leq 2|T| - 4.$
\label{Anti-degree-3}
\end{proposition}
\begin{proof}
For a finite tree $\Lambda$, set $m(\Lambda) := \Sigma_{v \in V(\Lambda)} (\mathrm{max} \{ \mathrm{deg}(v, \Lambda) - 3, 0 \})$. 
Then $\mathrm{deg}_{\mathrm{max}}(\Lambda) \geq 4$ if and only if $m(\Lambda) > 0$. 
Let $T$ be a finite tree such that $\mathrm{deg}_{\mathrm{max}}(T) \geq 4$, namely $m(T) > 0$. 
Then by applying Lemmas \ref{Computation_anti_degree}(1) and  \ref{Heteromorphic_lemma} to a vertex $u$ of $T$ with $\mathrm{deg}(u, T) \geq 4$, we obtain a finite tree $T^*$ such that
$$|T^*| = |T|+ 1, \ m(T^*) = m(T) -1, \ G(T) \hookrightarrow G(T^*) .$$
Hence, by repeating this argument, we obtain a finite tree $T'$ such that 
$$|T'| = |T|+ m(T), \ m(T') = 0 \ (\mbox{in particular, } \mathrm{deg}_{\mathrm{max}}(T') = 3), \ G(T) \hookrightarrow G(T').$$
In the remainder, we show $m(T) \leq |T| - 4$. 
Pick a vertex $v_0$ of $T$ with $\mathrm{deg}(v_0, T) \geq 4$ and let $T_0$ be a sub-tree of $\mathrm{St}(v_0, T)$, induced by $v_0$ and three vertices adjacent to $v_0$. 
Note that $|T_0| - 4 = 1 = m(T_0)$. 
The tree $T$ is obtained from $T$ by successively adding $|T| - 4$ edges to $T_0$. 
Since each added edge contributes to the number $m (\cdot)$ at most by $1$, we have the desired inequality $m(T) \leq |T| - 4$. 
\end{proof}

\begin{proof}[{\bf Proof of Theorem \ref{Anti-trees-deg3}}.]
Suppose that $\Lambda$ is a finite graph. 
In the case where $\Lambda$ is a tree of maximum degree $\leq 3$, the assertion is trivial. 
If $\Lambda$ is a tree of maximum degree $>3$, then we obtain the desired result by Proposition \ref{Anti-degree-3}. 
Hence, we may assume that $\Lambda$ is not a tree. 
Then by Theorem \ref{Anti-trees-theorem} due to Lee-Lee \cite{Lee-Lee}, there exists a finite tree $T$ such that $G(\Lambda) \hookrightarrow G(T)$ and that $|T| \leq |\Lambda| \cdot 2^{(|\Lambda| -1)}$. 
We now use Proposition \ref{Anti-degree-3}. 
Then the resulting finite tree $T'$ satisfies the desired property. 
\end{proof}

\section{Further discussions \label{final}}
In this section, we first discuss the following question due to S. Lee \cite{Lee}. 

\begin{question} 
For any graph $\Lambda$, is it possible that $G(\Lambda) \hookrightarrow G(P_n)$ for some $n$?
\label{question-anti-lines}
\end{question}

If $\Lambda$ satisfies $\mathrm{deg}_{\mathrm{max}}(\Lambda) \leq 2$, then $G(\Lambda) \hookrightarrow G(P_n)$ for some $n$ by Kim-Koberda's theorem \cite[Theorem 3.5]{Kim-Koberda-2} or Lee-Lee's theorem (Theorem \ref{Anti-trees-theorem}). 
By this fact and Theorem \ref{Anti-trees-deg3}, the above Question \ref{question-anti-lines} is reduced to the case when $\Lambda$ is a finite tree $T$ of maximum degree $3$. 
By using subdivision technique (see Lemma \ref{Complementary-subdivision}), we can further reduce Question \ref{question-anti-lines} to the case when $\Lambda = T$ satisfies the following condition. 
\begin{enumerate}
 \item[(C)] $\forall u \in V_3(T)$, $V(\mathrm{Lk}(u, T)) \subset V_2(T)$. 
\end{enumerate}
This condition says that $T$ is locally as illustrated in Figure \ref{d}(1). 
So, I would like to propose the following question. 

\begin{question}
For a finite tree $T$ satisfying the condition (C), does $G(T)$ embed into $G(P_n)$ for some $n$? 
In particular, is it possible that $G(T_0) \hookrightarrow G(P_n)$ for the tree $T_0$ in Figure \ref{d}(2) and some $n$?
\end{question}

\begin{figure}
\centering
\includegraphics[scale=0.24,clip]{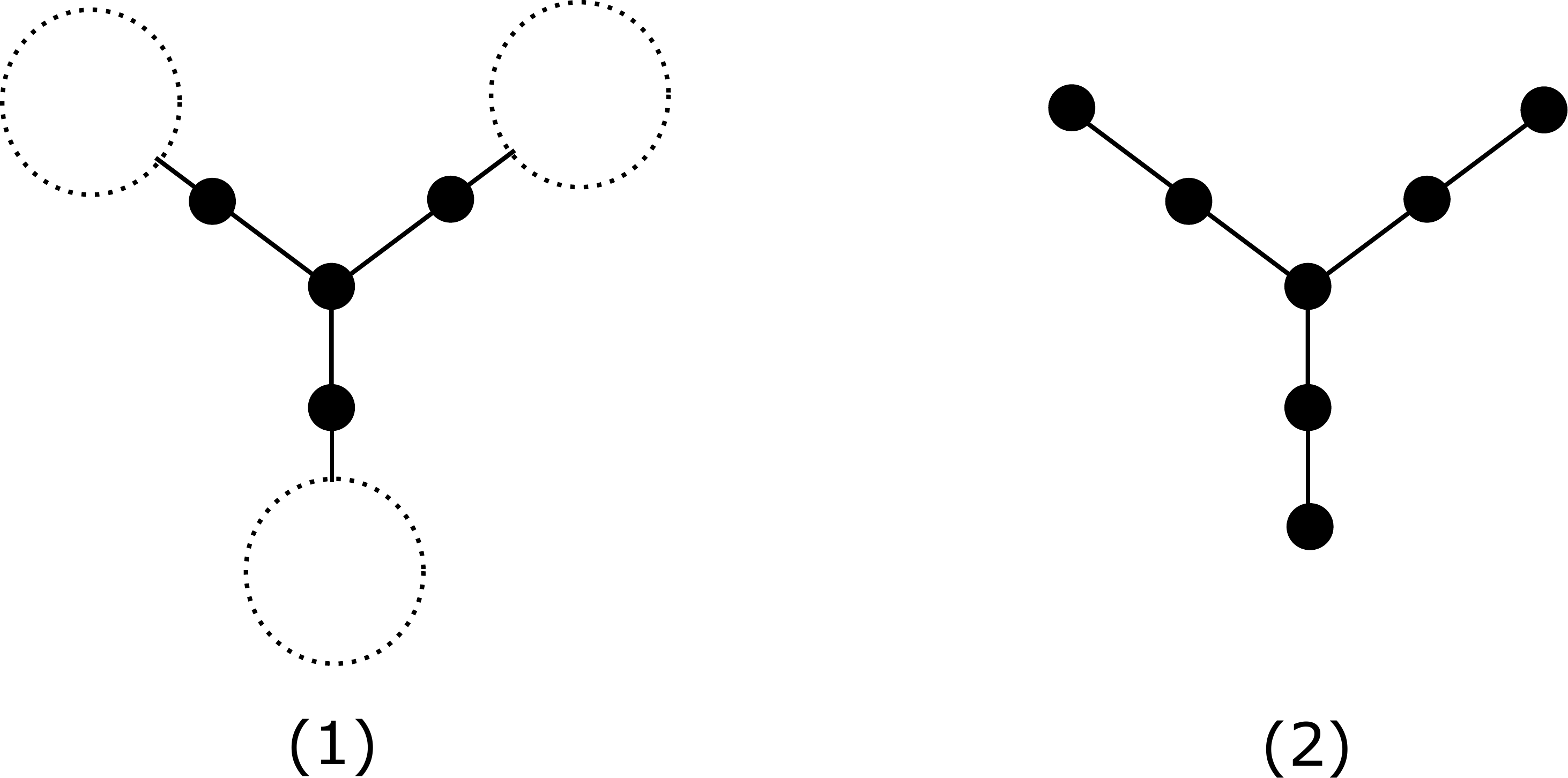}
\caption{
(1) A local picture of a tree $T$ satisfying the condition (C). 
(2) The graph $T_0$. 
}
\label{d}
\end{figure}

We next discuss relation between Corollary \ref{mapping-class-groups} and the following two known combinatorial tests for embedding RAAGs into mapping class groups. 
\begin{enumerate}
 \item[$\bullet$] The colouring test: 
 Kim and Koberda \cite{Kim-Koberda-3} proved that there exists a finite number $M$ (which depends on $\Sigma_{g, n}$) such that for any finite graph $\Lambda$ with $A(\Lambda) \hookrightarrow \mathcal{M}(\Sigma_{g,n})$, the chromatic number of $\Lambda$ does not exceed $M$. 
Note that this $M$ must be greater than or equal to the chromatic number of $\mathcal{C}(\Sigma_{g,n})$.
 \item[$\bullet$] The ``nested complexity length" test:  Bering IV, Conant and Gaster \cite{Bering-Conant-Gaster}, introduced the ``nested complexity length" of a graph and proved that for any finite graph $\Lambda$ with $A(\Lambda) \hookrightarrow \mathcal{M}(\Sigma_{g,n})$, the ``nested complexity length" of $\Lambda$ does not exceed $6g -6 + 2n$. 
\end{enumerate}
By using Corollary \ref{mapping-class-groups}, we prove the following proposition which shows that there exist RAAGs which cannot be embedded into $\mathcal{M}(\Sigma_{g,n})$, though they pass the above colouring and nested complexity length tests. 

\begin{proposition}
For any pair $(g, n)$ with $2 - 2g - n <0$, there exists a finite graph $\Lambda$ such that $A(\Lambda)$ cannot be embedded into $\mathcal{M}(\Sigma_{g,n})$, though the chromatic number of $\Lambda$ is not more than that of $\mathcal{C}(\Sigma_{g,n})$ and the nested complexity length of $\Lambda$ is not more than that of $\mathcal{C}(\Sigma_{g,n})$. 
\label{two_obst}
\end{proposition}

To prove this, we prepare the following lemma. 

\begin{lemma}
Let $K_{r}(2)$ be the complete $r$-partite graph of order $2$ (i.e., the complement of the disjoint union of $r$ copies of $P_2$). 
Then $A(K_{r}(2)) \hookrightarrow \mathcal{M}(\Sigma_{g,n})$ if and only if $r \leq g + \lfloor \frac{g + n}{2} \rfloor  -1$. 
\label{r-partite}
\end{lemma}
\begin{proof}
By \cite[Lemma 30]{Bering-Conant-Gaster}, $K_{r}(2)  \leq \mathcal{C}(\Sigma_{g,n})$ if and only if $r \leq g + \lfloor \frac{g + n}{2} \rfloor  -1$. 
Hence, we obtain the desired result by using Koberda's embedding theorem and Corollary \ref{mapping-class-groups}. 
\end{proof}

\begin{proof}[{\bf Proof of Proposition \ref{two_obst}}.]
For $r= g + \lfloor \frac{g + n}{2} \rfloor$, $A(K_r(2))$ cannot be embedded into $\mathcal{M}(\Sigma_{g,n})$ by Lemma \ref{r-partite}. 
However, as in the proof of \cite[Corollary 16]{Bering-Conant-Gaster}, we can easily see that the chromatic number (resp.~the nested complexity length) of $K_{r}(2)$ is equal to $r$ (resp.~ $2r$) and the chromatic number (resp.~the nested complexity length) of $\mathcal{C}(\Sigma_{g,n})$ is not less than $3g - 3 + n$ (resp.~ $6g- 6 + 2n$). 
Set $\Lambda:= K_r(2)$. 
\end{proof}

\section{Acknowledgement}
The author thanks my supervisor, Makoto Sakuma for carefully reading the first draft and suggesting a number of improvements. 
The author also thanks Erika Kuno for giving him helpful comments and telling him a question on embeddings of RAAGs into mapping class groups \cite[Question 1.1]{Kim-Koberda-3}. 
Moreover, the author owes Proposition \ref{two_obst}  to Thomas Koberda who informed the author of the paper due to Bering IV, Conant and Gaster \cite{Bering-Conant-Gaster}.


\begin{thebibliography}{99}

\bibitem{Agol}
I. Agol, `The virtual Haken conjecture', {\em Doc. Math.} 18 (2013) 1045--1087. 

\bibitem{Bering-Conant-Gaster}
E. Bering IV, G. Conant, J. Gaster, `On the complexity of finite subgraphs of the curve graph', preprint (2016), available at arXiv:1609.02548. 

\bibitem{Birman-Lubotzky-McCarthy}
J. Birman, A. Lubotzky, J. McCarthy, `Abelian and solvable subgroups of the mapping class groups', {\em Duke Math. J.} 50 (1983) 1107--1120. 

\bibitem{Casals-Ruiz}
M. Casals-Ruiz, `Embeddability and universal equivalence of partially commutative groups', {\em Int. Math. Res. Not.} (2015) 13575--13622. 

\bibitem{Casals-Ruiz-Duncan-Kazachkov}
M. Casals-Ruiz, A. Duncan, I. Kazachkov, `Embedddings between partially commutative groups: two counterexamples', 
{\em J. Algebra} 390 (2013) 87--99. 

\bibitem{Charney-Vogtmann}
R. Charney, K. Vogtmann, `Finiteness properties of automorphism groups of right-angled
Artin groups', 
{\em Bull. Lond. Math. Soc.} 41 (2009) 94--102. 

\bibitem{Crisp-Sageev-Sapir}
J. Crisp, M. Sageev, M. Sapir, `Surface subgroups of right-angled Artin groups', 
{\em Internat. J. Algebra Comput. } 18 (2008) 443--491. 

\bibitem{Droms}
C. Droms, 
`Graph groups, coherence, and three-manifolds',
{\em  J. Algebra}, 106 (1987) 484--489.

\bibitem{Kim}
S. Kim, `Co-contractions of graphs and right-angled Artin groups
', {\em Algebr. Geom. Topol.} 8 (2008) 849--868.

\bibitem{Kim-Koberda-1}
S. Kim, T. Koberda, `Embedability between right-angled Artin groups', {\em
Geom. Topol.} 17 (2013) 493--530.

\bibitem{Kim-Koberda-3}
S. Kim, T. Koberda, `An obstruction to embedding right-angled Artin groups
in mapping class groups', {\em
Int. Math. Res. Not.} 2014 (2014) 3912--3918. 

\bibitem{Kim-Koberda-2}
S. Kim, T. Koberda, `Anti-trees and right-angled Artin subgroups of braid groups', {\em
Geom. Topol.} 19 (2015) 3289--3306. 

\bibitem{Koberda}
T. Koberda, `Right-angled Artin groups and a generalized isomorphism problem for finitely generated subgroups of mapping class groups', {\em
Geom. Funct. Anal.} 22 (2012) 1541--1590. 

\bibitem{Lee-Lee}
E. Lee, S. Lee, `Path lifting properties and embedding between RAAGs', {\em J. Algebra} 448 (2016) 575--594. 

\bibitem{Lee}
S. Lee, Talk at the conference `The 11th East Asian School of Knots and Related Topics', January 28 (2016). 

\bibitem{Wise}
D. Wise, `The structure of groups with a  quasiconvex hierarchy', preprint (2011), available at http://www.math.mcgill.ca/wise/papers.html. 
\end{thebibliography}
\end{document}